\documentclass[final,leqno]{siamltex}

\usepackage{graphicx} 
\usepackage{epsfig} 
\usepackage{times} 
\usepackage{amsmath} 
\usepackage{amssymb}  
\usepackage{amsfonts}
\usepackage{epstopdf}
\usepackage[caption=false]{subfig}
\usepackage[sort]{cite}
\usepackage[colorlinks=true]{hyperref}
\usepackage{comment}
\usepackage{mathtools}
\usepackage{eurosym}
\usepackage{multirow}
\usepackage{tikz}
\usetikzlibrary{arrows}

\newtheorem{remark}[theorem]{Remark}
\newtheorem{assumption}[theorem]{Assumption}
\newtheorem{problem}[theorem]{Problem}
\newtheorem{algorithm}[theorem]{Algorithm}
\newcommand{\sr}{\stackrel}

\newcommand{\tri}{\sr{\triangleq}{=}}

\newcommand{\be}{\begin{equation}}
\newcommand{\ee}{\end{equation}}
\newcommand{\bea}{\begin{eqnarray}}
\newcommand{\eea}{\end{eqnarray}}
\newcommand{\bes}{\begin{eqnarray*}}
\newcommand{\ees}{\end{eqnarray*}}
\newcommand{\bi}{\begin{itemize}}
\newcommand{\ei}{\end{itemize}}
\newcommand{\ben}{\begin{enumerate}}
\newcommand{\een}{\end{enumerate}}
\newcommand{\bp}{\begin{problem}}
\newcommand{\ep}{\end{problem}}
\newcommand{\hso}{\hspace{.1in}}
\newcommand{\hst}{\hspace{.2in}}

\newcommand{\hse}{\hspace{.05in}}

\newcommand{\argmin}{\arg\!\min}

\title{Infinite Horizon Average Cost Dynamic Programming Subject to Total Variation Distance Ambiguity}

\author{Ioannis~Tzortzis\thanks{Department of Electrical and Computer Engineering, University of Cyprus (UCY), Nicosia, Cyprus. ({\tt tzortzis.ioannis@ucy.ac.cy}).}
        \and Charalambos~D.~Charalambous\thanks{Department of Electrical and Computer Engineering, University of Cyprus (UCY), Nicosia, Cyprus. ({\tt chadcha@ucy.ac.cy}).}
\and Themistoklis~Charalambous\thanks{Department of Signals and Systems, Chalmers University of Technology, Gothenburg, Sweden. ({\tt themistoklis.charalambous@chalmers.se}).}}

\begin{document}
\maketitle

\begin{abstract}
We analyze the infinite horizon minimax average cost Markov Control Model (MCM), for a class of controlled process conditional distributions, which belong to a ball, with respect to total variation distance metric, centered at a known nominal controlled conditional distribution with radius $R\in [0,2]$, in which the minimization is over the control strategies and the maximization is over conditional distributions. Upon performing the maximization,   a dynamic programming equation is obtained which includes, in addition to the standard terms, the oscillator semi-norm of the cost-to-go. 

First, the dynamic programming equation is analyzed for finite state and control spaces. We show that if the nominal controlled process distribution is irreducible, then for every stationary Markov control policy the maximizing conditional distribution of the controlled process is also  irreducible for $R \in [0,R_{max}]$. \\
Second, the generalized dynamic programming is analyzed for Borel spaces. We derive necessary and sufficient conditions for  any control strategy to be optimal.

Through our analysis, new dynamic programming equations and new policy iteration algorithms are derived. The main feature of the new policy iteration algorithms (which are applied for finite alphabet spaces) is that the policy evaluation and policy improvement steps  are performed by using the maximizing conditional distribution, which is obtained via a water filling solution.
Finally,  the application of the new dynamic programming equations and the corresponding policy iteration algorithms are shown via illustrative examples.

\end{abstract}



\pagestyle{myheadings}
\thispagestyle{plain}
\markboth{I. TZORTZIS, C. D. CHARALAMBOUS AND T. CHARALAMBOUS}{Infinite Horizon Average Cost Dynamic Programming}

\section{Introduction}\label{sec.intro}
The infinite horizon average cost per unit-time discrete-time Markov Control Model (MCM), with deterministic strategies is analysed, in an anthology of papers \cite{Arapostathis93,Borkar84,Borkar88,Borkar89,Sennott95}.
In such MCMs, the corresponding cost-to-go and the dynamic programming recursions depend on the conditional distribution of the underlying controlled process \cite{caines88}. This means, any ambiguity of the controlled process conditional distribution will affect the optimality and robustness of the optimal decision strategies.

In this paper, we investigate the effects of any ambiguity in the controlled process conditional distribution on the cost-to-go and dynamic programming. We model the ambiguity in the controlled conditional distributions by a ball with respect to the total variation distance metric, centered at a known nominal controlled conditional distribution with radius $R\in[0,2]$. Then, we re-formulate the infinite horizon average cost MCM using minimax optimization techniques, in which the control strategy seeks to minimize the payoff while the conditional distribution, from the class of  total variation distance ball, seeks to maximize it.

 We begin our analysis by first considering MCM's defined on finite state and control spaces. By employing certain results from \cite{ctlthem2013j}, we obtain the characterization of  the maximizing conditional distribution and  the corresponding dynamic programming equation.  The main feature of the maximizing conditional distribution is its  characterization via a water-filling solution,  which is similar in spirit, to extremum problems encountered in information theory, such as,  channel capacity and lossy data compression \cite{cover}. This leads to a dynamic programming equation, which includes in its right hand side, the oscillator semi-norm of the cost-to-go or value function, in addition to the standard terms. We show that, if the nominal controlled process distribution is irreducible, then for every stationary Markov control policy the maximizing conditional distribution of the controlled process is also  irreducible, the optimal control strategies exists,  for $R \in [0,R_{max}]$. Moreover, for this range of $R$, we derive a new policy iteration algorithm.   
 
Subsequently, we consider general Borel spaces, we invoke  a pair of dynamic programming equations (called generalized), and we derive necessary and sufficient conditions of optimality,  based on the concept of canonical triplets  \cite{dynkin79,Schal92,hern1996discrete,puterman94}. This treatment characterizes optimal strategies  for any ball of radius $R\in[0,2]$.   The main feature of the corresponding policy iteration algorithm (which  is applied for finite alphabet spaces), is that the policy evaluation and policy improvement steps  are performed using the maximizing conditional distribution.
 
The remainder of the paper is organized as follows. In Section \ref{subsec.discuss.main.res}, we introduce the classical infinite horizon dynamic programming equation of MCM with an average cost per unit-time optimality criterion, and we briefly discuss the main results derived in the paper. In Section \ref{sec.minimaxstochastic.control}, we give some preliminary results concerning the maximization of a linear functional subject to total variation distance. In Section \ref{sec.minimaxstoch.control}, we study the infinite horizon average cost Markov decision problem for finite state and control spaces, and we derive a new dynamic programming recursion and the corresponding policy iteration algorithm. In Section \ref{subsec.inf_ave_gener_DP}, we consider general Borel spaces, and we investigate  the infinite horizon average cost Markov decision problem, using  the generalized dynamic programming equations. We also introduce a generalized policy iteration algorithm when the state and control spaces are of finite cardinality. In Section \ref{sec.examples}, we present two examples which illustrate the implications of the the new dynamic programming recursions on the corresponding policy iteration algorithms.

\subsection{Discussion on the Main Results}\label{subsec.discuss.main.res}
In this section, we describe the main results obtained in the paper with respect to the existing literature. Since we treat finite alphabet spaces and Borel spaces, the formulation below is introduced for Borel spaces.

\subsubsection{Dynamic Programming of Infinite-Horizon MCM} An infinite horizon MCM with deterministic strategies is a five-tuple
\begin{equation}\label{inf.hor.mcm.def}
\Big( {\cal X},  {\cal U}, \{ {\cal U}(x): x\in {\cal X}\},\{ Q(dz| x, u){:} (x, u) {\in} {\cal X} {\times} {\cal U} \},f\Big)
\end{equation}
consisting of the following.
\ben
\item[a)] State Space. A complete separable metric space (called a Polish space) ${\cal X}$, which models the state space of the controlled random process $\{x_k\in {\cal X}:k\in \mathbb{N}\}$, $\mathbb{N}\triangleq 0,1,\dots$.
\item[b)] Control or Action Space. A Polish space $ {\cal U}$, which models the control or action set of the control random process $\{u_k\in {\cal U}:k\in \mathbb{N}\}$.
\item[c)] Feasible Controls or Actions. A family $\{ {\cal U}(x): x \in {\cal X}\}$ of non-empty measurable subsets ${\cal U}(x)$ of ${\cal U}$, where ${\cal U}(x)$ denotes the set of feasible controls or actions, when the controlled process is in state $x \in {\cal X}$, and  the feasible state-actions pairs are measurable subsets of  ${\cal X} \times {\cal U}$, defined by \begin{equation}\label{feas.state.action.pair}
{\mathbb K} \triangleq \Big\{ (x, u): x \in {\cal X}, u \in {\cal U}(x)\Big\}. \end{equation} 
\item[d)] Controlled Process Distribution. A conditional distribution or stochastic kernel $Q(dz|x, u)$ on ${\cal X}$
given $(x, u) \in {\mathbb K} \subseteq {\cal X} \times {\cal U}$, which corresponds to the controlled process transition probability distribution.
\item[e)] One-Stage-Cost. A non-negative measurable function $f: {\mathbb K} \longmapsto [0, \infty]$, called the one-stage-cost, such that $f(x,\cdot)$ does not take the value $+\infty$ for each $x \in {\cal X}$. 
\een

To ensure the existence of measurable controls we make the following assumption.
\begin{assumption} \cite{hern1996discrete} ${\mathbb K}$ contains the graph of a measurable functions from ${\cal X}$ to ${\cal U}$; that is, there is a measurable function $\varphi:{\cal X}\longmapsto {\cal U}$ such that $\varphi(x)\in {\cal U}(x)$, for all $x\in {\cal X}$. The set of all such functions denoted by $\mathbb{F}$ are called selectors of the multifunction $x\longmapsto {\cal U}(x)$.
\end{assumption}

 We equip the spaces ${\cal X}$ and ${\cal U}$ with the natural $\sigma$-algebra ${\cal B}({\cal X})$ and ${\cal B}({\cal U})$, respectively. For any measurable spaces $({\cal X}, {\cal B}({\cal X})), ({\cal U}, {\cal B}({\cal U}))$, we denote the set of stochastic Kernels on $({\cal X}, {\cal B}({\cal X}))$  conditioned on $({\cal U}, {\cal B}({\cal U}))$ by ${\cal Q}({\cal X}|{\cal U})$, and we denote the set of probability distributions on $({\cal X}, {\cal B}({\cal X}))$ by ${\cal M}_1({\cal X})$. Next, we give the definition of deterministic stationary Markov control policies.

\begin{definition}
A deterministic stationary Markov control policy is a measurable function (selector) $g:{\cal X}\longmapsto {\cal U}$ such that $g(x_t)\in {\cal U}(x_t)$, $\forall x_t\in {\cal X}$, $t=0,1,\dots$. The set of such deterministic stationary Markov policies is denoted by $G_{SM}$, and the set of all deterministic control policies (i.e., non-stationary, non-Markov) is denoted by $G$.
\end{definition}

Define the $n$-stage expected cost, for a fixed $x_0=x$, by
\begin{equation}\label{nstage.cost}
J^o_n(g,x)\triangleq\mathbb{E}_x^{g}\Big\{\sum_{k=0}^{n-1}f(x_k,u_k)\Big\}
\end{equation}
where $\mathbb{E}_x^{g}\{\cdot\}$ indicates the dependence of the expectation operation on the policy $g\in G$ and $x_0=x$. Then, the average cost per unit-time when policy $g\in G$ is used, given $x_0=x$, is defined by
\begin{equation}\label{prel.perf.criter}
J^o(g,x)\triangleq\limsup_{n\rightarrow \infty}\frac{1}{n}J^o_n(g,x).
\end{equation}
The Markov Control Problem (MCP) is to find a control policy $g^*\in G$ such that
\begin{equation}\label{markov.dec.problem}
J^o(g^*,x)\triangleq\inf_{g\in G}J^o(g,x)=J^{o,*}(x),\quad \forall x\in {\cal X}.
\end{equation}

For finite cardinality spaces $({\cal X},{\cal U})$, it is known \cite{varayia86,puterman94,hern1996discrete,vanSchuppen10}, that if $f$ is bounded and for all stationary Markov control policies $g\in G_{SM}$ the transition probability matrix $Q(z|x,u)$ is irreducible (that is, all stationary policies have at most one recurrent class), then there exists a solution $V^o:{\cal X}\longmapsto \mathbb{R}$ and a constant (independent of $x\in{\cal X}$) $J^{o,*}\in \mathbb{R}$ such that $(J^{o,*},V^o(x))$ is the solution of the dynamic programming (of the infinite-horizon MCP \eqref{markov.dec.problem})
\begin{equation}\label{prel.perf.criter.dp}
J^{o,*}+V^o(x)=\inf_{u\in {\cal U}(x)}\Big\{f(x,u)+\sum_{z\in {\cal X}}Q(z|x,u)V^o(z)\Big\}
\end{equation}
from which existence of optimal policy $g^*\in G_{SM}$ is obtained. However, if the irreducibility condition is not satisfied (i.e., there is more than one recurrent class), then the dynamic programming equation \eqref{prel.perf.criter.dp} may not be sufficient to give the optimal policy and the minimum cost\cite{puterman94,hern1996discrete}. In this case, \eqref{prel.perf.criter.dp} is replaced by the following 
\begin{subequations}\label{cla.DP}
\begin{align}
J^{o,*}(x)&=\inf_{u\in {\cal U}(x)}\Big\{\sum_{z\in{ \cal X}}Q(z|x,u)J^{o,*}(z)\Big\}\label{cla.DP.a}\\
J^{o,*}(x)+V^o(x)&=\inf_{u\in {\cal U}(x)}\Big\{f(x,u)+\sum_{z\in{\cal X}}Q(z|x,u)V^o(z)\Big\}.\label{cla.DP.b}
\end{align}
\end{subequations}
We refer to \eqref{cla.DP.a} as the first general dynamic programming equation and to \eqref{cla.DP.b} as the second general dynamic programming equation\footnote{Some authors use the term multichain, instead.}. Note that, the pair of generalized dynamic programming equations \eqref{cla.DP.a}-\eqref{cla.DP.b} solves the MCP \eqref{markov.dec.problem}, without imposing irreducibility of the conditional distribution of the controlled process \cite{puterman94,hern1996discrete}. Similar results are also known for Borel spaces, by replacing summations in \eqref{prel.perf.criter.dp} and \eqref{cla.DP} by integrals with respect to conditional distribution, while the characterization of the existence of optimal policies is done via canonical triplets \cite{hern1996discrete}.

Since the MCP \eqref{markov.dec.problem} and the dynamic programming equation \eqref{prel.perf.criter.dp} are functionals of the conditional distribution of the controlled process, then the optimal strategies $g\in G$ are obtained based on the assumption of having an accurate knowledge of the conditional distribution $Q(dz|x,u)$.  Hence, any ambiguity or mismatch of $Q(dz|x,u)$ from the true conditional distribution will affect the optimality of the strategies.

Motivated by this implication, in this paper we consider the problem discussed in the next section.

\subsubsection{Dynamic Programming of Infinite-Horizon MCM with Total Variation Distance Ambiguity}
Recall the total variation distance between two probability measures, $|| \cdot ||_{TV}: {\cal M}_1({\cal X}) \times {\cal M}_1({\cal X})\longmapsto [0,\infty]$, defined by 
\begin{equation*}
||\alpha-\beta ||_{TV} \triangleq \sup_{ P \in {\cal P}({\cal X})} \sum_{F_i \in P} |\alpha (F_i)- \beta(F_i)|,\hso \alpha, \beta\in {\cal M}_1({\cal X})
\end{equation*}
where ${\cal P}({\cal X})$ denotes the collection of all finite partitions of ${\cal X}$. 

In this paper, we will derive the analogues of \eqref{prel.perf.criter.dp} and \eqref{cla.DP.a}-\eqref{cla.DP.b}, for the class of conditional distributions of the controlled process 
$Q(dz|x,u)$, $(x,u)\in \mathbb{K}$ which are stationary, and belong to a ball with respect to total variation distance metric, centered at a nominal controlled process distribution $Q^o(dz|x,u)$, $(x,u){\in} \mathbb{K}$, having radius $R\in [0,2]$ (specifically, $\{Q(dz|x,u):||Q(\cdot|x,u)-Q^o(\cdot|x,u)||_{TV}\leq R\}$).

The precise definition is the following.
\begin{definition}\label{dt} 
For each $g\in G_{SM}$, the nominal controlled process $\{x^g_t: t=0,1,\dots\}$ has a stationary conditional distribution defined by 
\begin{equation*}
Prob(x_{t} \in A| x^{t-1},u^{t-1}   )\triangleq Q^o(A| x_{t-1}, u_{t-1}), \ \forall A \in {\cal B}({\cal X}),\ t=0,1,\dots
\end{equation*}
where $Q^o(\cdot| \cdot, \cdot) \in {\cal Q}({\cal X} | \mathbb{K})$. Given the nominal controlled process and  $R \in [0,2]$, the true controlled process conditional distributions are stationary, and belong to the total variation distance ball defined by 
\begin{equation}\label{ambiguityclass.controlled}
  \mathbf{B}_{R}(Q^o) (x,u) \triangleq \Big\{ { Q}(\cdot| x,u) \in {\cal M}_1({\cal X}):|| { Q}(\cdot|  x,u) -Q^o(\cdot|  x,u) ||_{TV}  \leq R \Big\}, 
  (x,u)\in \mathbb{K}. 
\end{equation}
\end{definition}

Next, we consider the analogue of \eqref{markov.dec.problem}. Define the $n$-stage expected cost by
\begin{equation}
J_n(g,Q,x)\triangleq\mathbb{E}_x^{g}\Big\{\sum_{k=0}^{n-1}f(x_k,u_k)\Big\}
\end{equation}
and the corresponding maximizing $n$-stage expected cost by
\begin{equation}\label{payoff.int1}
J_n(g,x)\triangleq \sup_{  {{ Q}(\cdot| x, u)  \in \mathbf{B}_{R}(Q^o)(x,u)}}\mathbb{E}_x^{g}\Big\{\sum_{k=0}^{n-1}f(x_k,u_k)\Big\}.
\end{equation}
Then, the maximizing average cost per unit-time when policy $g\in G$ is used, given $x_0=x$, is defined by
\begin{equation}\label{payofff}
J(g,x)\triangleq\limsup_{n\rightarrow\infty} \frac{1}{n}J_n(g,x).
\end{equation}

The minimax MCP subject to ambiguity defined by \eqref{ambiguityclass.controlled}, is to choose a control policy $g^*\in G$ such that

\begin{equation}\label{markov.dec.problem.amb}
J(g^*,x)\triangleq\inf_{g\in G}J(g,x)=J^*(x),\quad \forall x\in {\cal X}.
\end{equation} 
A conditional distribution $Q^*$ that satisfies \eqref{payofff} (see also \eqref{payoff.int1}) is called a maximizing conditional distribution, a policy $g^*$ that satisfies \eqref{markov.dec.problem.amb} is called an average cost optimal policy, and the corresponding $J^*(\cdot)$ is the minimum cost or value function of the minimax MCP.
%

Next, we introduce an assumption for the minimax MCP defined by \eqref{markov.dec.problem.amb}. 
\begin{assumption}\label{assump.cost}
\ben \item[(a)] The map $f:{\cal X}\times {\cal U}\longmapsto \mathbb{R}$ is bounded, continuous and non-negative.
\item[(b)] The set ${\cal U}(x)$ is compact for all $x\in {\cal X}$.
\item[(c)] The map $Q^o(A|\cdot,\cdot)$ is  continuous on $\mathbb{K}$ for every Borel set.
\een
\end{assumption}

Note that it is possible to relax Assumption \ref{assump.cost}, for example, $f(x,\cdot)$ can be replaced by a lower semi-continuous function on ${\cal U}(x)$ for every $x\in {\cal X}$, which is non-negative (see \cite{hern1996discrete} for several relaxations).

We derive the following results.

\textbf{Dynamic Programming Equations for Finite Alphabet Spaces.}
In Section \ref{sec.minimaxstoch.control}, we assume that $({\cal X}, {\cal U})$ are of finite cardinality and we show that if for all stationary Markov control policies $g\in G_{SM}$, and for a given total variation parameter $R\in [0,2]$, the maximizing transition probability matrix $Q^*(g)$ is irreducible, then the dynamic programming equation corresponding to minimax MCP \eqref{markov.dec.problem.amb} is given by  
\begin{equation}\label{n.DP7}
J^*+V(x)=\min_{u\in\cal U}\Big\{f(x,u)
+\sum_{z\in\cal X}Q^o(z|x,u)V(z)
+\frac{R}{2}\big(\sup_{z\in\cal X}V(z)-\inf_{z\in\cal X}V(z)\big)\Big\}.
\end{equation}
The new term entering in the right side of \eqref{n.DP7} is the oscillator semi-norm of the future pay-off. 

\textbf{Generalized Dynamic Programming Equations for Borel Spaces.}
In Section \ref{subsec.inf_ave_gener_DP}, we assume that $({\cal X},{\cal U})$ are Borel spaces, and we utilize the concept of canonical triplets to establish existence of optimal strategies via the following generalized dynamic programming equations
\begin{subequations}\label{cla.DPro}
\begin{eqnarray}
J^*(x)&{=}&\inf_{u\in {\cal U}(x)}\Big\{\int_{\mathclap{{\ \cal X}}}Q^o(dz|x,u)J^*(z) {+}\frac{R}{2}\big(\sup_{z\in\cal X}J^*(z){-}\inf_{z\in\cal X}J^*(z)\big)\Big\}\label{cla.DPro.a}\\
\quad\qquad J^*(x){+}V(x)&{=}&\inf_{u\in {\cal U}(x)}\Big\{f(x,u){+}\int_{\mathclap{{\ \cal X}}}Q^o(dz|x,u)V(z) {+}\frac{R}{2}\big(\sup_{z\in\cal X}V(z){-}\inf_{z\in\cal X}V(z)\big)\Big\}.\label{cla.DPro.b}
\end{eqnarray}
\end{subequations}
Since Borel spaces include finite alphabet spaces, if  irreducibility condition is violated, then the existence of optimal strategies is characterized by the finite alphabet version of \eqref{cla.DPro.a}-\eqref{cla.DPro.b}.

In addition, we obtain the following.
\ben
\item Characterize the maximizing conditional distribution corresponding to the supremum in \eqref{payofff}.
\item Derive new policy iteration algorithms (applied for finite alphabet spaces), in which the policy evaluation and the policy improvement steps are performed by using the maximizing conditional distribution obtained under total variation distance ambiguity constraint.
\een
Finally, in Section \ref{sec.examples} we present illustrative examples based on \eqref{n.DP7} and \eqref{cla.DPro}.

%
%
%
%

\section{Maximization over Total Variation Distance Ambiguity}\label{sec.minimaxstochastic.control}
In this section, we recall certain results from \cite{ctlthem2013j}, concerning the characterization of the extremum problem of maximizing a linear functional subject to total variation distance ambiguity. We use these results to derive the new dynamic programming equations. 

Let $({\cal X},d_{\cal X})$ denote a complete, separable metric space (a Polish space), and $({\cal X}, {\cal B}({\cal X}))$ the corresponding measurable space, in which ${\cal B}({\cal X})$ is the $\sigma$-algebra generated by open sets in ${\cal X}$. Define the spaces
\begin{eqnarray*}
&&BC({\cal X})\triangleq \big\{\mbox{Bounded continuous functions}\hse\ell:{\cal X}\longmapsto{\mathbb R}: ||\ell||\triangleq\sup_{x\in{\cal X}}|\ell(x)|<\infty\big\}\\
&&BC^+({\cal X})\triangleq \big\{\ell\in BC({\cal X}):\ell\geq 0\big\}.
\end{eqnarray*}
For $\ell\in BC^+({\cal X})$, and $\mu\in {\cal M}_1({\cal X})$ fixed, then we have 
\begin{equation}
 L(\nu^*)\triangleq \sup_{||\nu-\mu||_{TV}\leq R}\int_{{\cal X}}\ell(x)\nu(dx)=\frac{R}{2}  \Big\{ \sup_{x \in {{\cal X}}} \ell(x)  -  \inf_{x \in {\cal X}} \ell(x) \Big\}    + \int_{{\cal X}}\ell(x)\mu(dx)\label{f2n}
\end{equation} 
where $R\in[0,2]$, ${ \nu}^*$ satisfies the constraint $||{ \xi}^*||_{TV}= ||{ \nu}^*-{ \mu}||_{TV} =R$, it is normalized  ${ \nu}^*({\cal X})=1$, and $\nu^*(A) \in [0,1]$ on any $A \in {\cal B}({\cal X})$. If ${\cal X}$ is a compact set, since $\ell(\cdot)\in BC^+({\cal X})$ then both the supremum and infimum are attained and they are finite. Define\footnote{We adopt the standard definitions; infimum (supremum) of an empty set to be $+\infty$ ($-\infty$).}
\begin{align*}
x^0 \in {\cal X}^0  &\triangleq  \Big\{ x \in \overline{{\cal X}}: \ell(x) = \sup \{\ell(x): x\in {\cal X}\} \equiv \ell_{\max}\Big\} \\
x_0 \in {\cal X}_0  &\triangleq \Big\{ x \in \overline{{\cal X}}: \ell(x) = \inf\{\ell(x): x\in {\cal X}\} \equiv \ell_{\min} \Big\}
\end{align*}
where $\overline{{\cal X}}$ denotes the closure\footnote{Closure of a set ${\cal X}$ consists of all points in ${\cal X}$ plus the limit points of ${\cal X}$.} of ${\cal X}$. Then, the pay-off ${ L}({ \nu}^*)$ can be written as
\begin{equation}\label{n2}
L( \nu^*)
= \int_{{\cal X}^0}     \ell_{\max}   \nu^*(dx) +   \int_{{\cal X}_0}   \ell_{\min}     \nu^*(dx)+   \int_{{\cal X} \setminus {\cal X}^0\cup {\cal X}_0} \ell(x) \mu(dx)  
\end{equation}
and the optimal distribution ${\nu}^*\in {\cal M}_1 ({\cal X})$, which satisfies the total variation constraint, is given by
\begin{subequations}\label{cond1}
\begin{align}
\int_{{\cal X}^0} \nu^*(dx)&= \mu({\cal X}^0) + \frac{R}{2}\in[0,1]\\
 \int_{{\cal X}_0} \nu^*(dx)&= \mu({\cal X}_0) - \frac{R}{2}\in[0,1]  \\
 \nu^*(A)&= \mu(A), \hso \forall A \subseteq {\cal X} \setminus {\cal X}^0\cup {\cal X}_0. 
\end{align}\end{subequations}
Note that, if ${\cal X}^0$ is empty then $\nu^*({\cal X}^0)=R/2$ and if ${\cal X}_0$ is empty then $\nu^*({\cal X}_0)=0$.

Next, we elaborate on the form of the maximizing measure for finite and countable alphabet spaces,  and its water filling behavior, since we use them to analyze infinite horizon MCP with finite state and control spaces.\\

\subsection{The Maximizing Measure for Finite and Countable Alphabet Spaces}\label{subsec.finite.extr}

Let ${\cal X}$ be a non-empty denumerable set endowed with the discrete topology. If the cardinality of ${\cal X}$ denoted by $|{\cal X}|$ is finite, then we can identify any $x\in {\cal X}$ by a unit vector in $\mathbb{R}^{|{\cal X}|}$. Define the set of probability vectors on ${\cal X}$ by 
\begin{equation}
\mathbb{P}({\cal X})\triangleq \Big\{p=(p_1,\dots,p_{|{\cal X}|}):p(x)\geq 0,x=1,\dots,|{\cal X}|,\sum_{x\in{\cal X}}p(x)=1 \Big\}.
\end{equation}
That is, $\mathbb{P}({\cal X})$ is the set of all $|{\cal X}|$-dimensional vectors which are probability vectors $\{\nu(x):x\in{\cal X}\}\in \mathbb{P}({\cal X})$, $\{\mu(x):x\in{\cal X}\}\in\mathbb{P}({\cal X})$. Also, let $\ell \triangleq\{\ell(x):x\in{\cal X}\}\in \mathbb{R}_+^{|{\cal X}|}$ (i.e., the set of non-negative vectors of dimension $|{\cal X}|$). Then,  \eqref{f2n} may be written as follows 
\begin{equation}\label{eq.finite.payoff}
L(\nu^*)=\max_{\nu\in \mathbb{B}_R(\mu)}\sum_{x\in{\cal X}}\ell(x)\nu(x)
\end{equation}
where \begin{equation}
\mathbb{B}_R(\mu)\triangleq \Big\{\nu\in\mathbb{P}({\cal X}):||\nu-\mu||_{TV}\triangleq \sum_{x\in{\cal X}}|\nu(x)-\mu(x)|\leq R\Big\}.
\end{equation}

 By defining $\xi(x)\triangleq \nu(x)-\mu(x)$, $x=1,\dots,|{\cal X}|$, then $\sum_{x\in{\cal X}}\xi(x)=0$, and $||\xi||_{TV}=\xi^+({\cal X})+\xi^-({\cal X})$ denotes the total variation of $\xi$, where $\xi^+=\max\{\xi,0\}$ and $\xi^-=\max\{-\xi,0\}$ stand for the positive and negative part of $\xi$, respectively. Therefore,
\begin{equation*}
\sum_{x\in {\cal X}}\xi(x)=\sum_{x\in{\cal X}}\xi^+(x)-\sum_{x\in {\cal X}}\xi^-(x),\quad ||\xi||_{TV}=\sum_{x\in {\cal X}}|\xi(x)|=\sum_{x\in{\cal X}}\xi^+(x)+\sum_{x\in {\cal X}}\xi^-(x)
\end{equation*}
and hence $\sum_{x\in{\cal X}}\xi^+(x)\equiv\alpha/2\equiv\sum_{x\in{\cal X}}\xi^-(x)$. In addition, since
\begin{equation}
\sum_{x\in{\cal X}}\ell(x)\xi(x)=\sum_{x\in{\cal X}}\ell(x)\xi^+(x)-\sum_{x\in{\cal X}}\ell(x)\xi^-(x)
\end{equation}
then \eqref{eq.finite.payoff} can be reformulated as follows
\begin{equation}\label{eq.maxm.reform}
\max_{\nu\in \mathbb{B}_R(\mu)}\sum_{x\in{\cal X}}\ell(x)\nu(x)\longrightarrow \sum_{x\in{\cal X}}\ell(x)\mu(x)+\max_{\xi\in \widetilde{\mathbb{B}}_R(\mu)}\sum_{x\in{\cal X}}\ell(x)\xi(x)
\end{equation}
where $\xi\in \widetilde{\mathbb{B}}_R(\mu)$ is described by the constraints 
\begin{equation}
\alpha\triangleq \sum_{x\in{\cal X}}|\xi(x)|\leq R,\quad \sum_{x\in {\cal X}}\xi(x)=0,\quad 0\leq \xi(x)+\mu(x)\leq 1,\quad \forall x\in{\cal X}.
\end{equation}
The solution of \eqref{eq.maxm.reform} is obtained by first identifying the partition of ${\cal X}$ into disjoint sets $({\cal X}^0,{\cal X}\setminus{\cal X}^0)$, and then by finding upper and lower bounds on the probabilities of ${\cal X}^0$ and ${\cal X}\setminus{\cal X}^0$, which are achievable \cite{ctlthem2013j}.

Towards this end, define the maximum and minimum values of $\{\ell(x):x\in{\cal X}\}$ by  
\begin{equation}
\ell_{\max}\triangleq \max_{x\in{\cal X}}\ell(x),\quad \ell_{\min}\triangleq \min_{x\in{\cal X}}\ell(x)\nonumber\end{equation}
 and their corresponding support sets by
\begin{equation}
 {\cal X}^0 \triangleq \big\{x\in {\cal X}:\ell(x)=\ell_{\max} \big\},\quad {\cal X}_0\triangleq\big\{x\in {\cal X}:\ell(x)=\ell_{\min} \big\}.\nonumber
\end{equation}
For all remaining sequence, $\big\{\ell(x):x\in {\cal X} \setminus {\cal X}^0\cup{\cal X}_0\big\}$, and for $1\leq r\leq |{\cal X}\setminus {\cal X}^0\cup {\cal X}_0|$, define recursively the set of indices for which the sequence achieves its $(k+1)^{th}$ smallest value by 
\begin{equation}\label{Sigmaksets}
{\cal X}_k {\triangleq}\Big\{x{\in} {\cal X}{:}\ell(x){=}\min\big\{\ell(\alpha){:} \alpha \in {\cal X}\setminus {\cal X}^0\cup(\bigcup_{j=1}^k{\cal X}_{j-1})\big\} \Big\},\quad k\in\{1,2,\hdots,r\}
\end{equation}
 till all the elements of ${\cal X}$ are exhausted. Further, define the corresponding values of the sequence on sets ${\cal X}_k$ by
 \begin{equation} \label{ellSigmaksets}
 \ell({\cal X}_k)\triangleq\min_{x\in{\cal X}\setminus{\cal X}^0\cup(\bigcup_{j=1}^k{\cal X}_{j-1})}\ell(x),\hst k\in\{1,2,\hdots,r\}
 \end{equation}
 where $r$ is the number of ${\cal X}_k$ sets which is at most $|{\cal X}\setminus{\cal X}^0\cup{\cal X}_0|$. 

From \cite{ctlthem2013j} we have the following. The maximum pay-off subject to the total variation constraint is given by 
 \begin{equation} { L}( { \nu}^*)=\ell_{\max}\nu^*({\cal X}^0)+\ell_{\min}\nu^*({\cal X}_0)+\sum_{k=1}^r\ell({\cal X}_k)\nu^*({\cal X}_k). \label{mpoff1}   
 \end{equation}
Moreover, the optimal probabilities are given by the following equations (water-filling solution).
\begin{subequations}\label{all3}
\begin{align}
&\nu^*({\cal X}^0)= \mu({\cal X}^0)+\frac{\alpha}{2}\label{all3a}\\
&\nu^*({\cal X}_0)= \Big(\mu({\cal X}_0)-\frac{\alpha}{2}\Big)^+\label{all3b}\\
&\nu^*({\cal X}_k)= \Big(\mu({\cal X}_k)-\Big(\frac{\alpha}{2}-\sum_{j=1}^k\mu({\cal X}_{j-1})\Big)^+\Big)^+\label{all3c}\\
&\alpha = \min\Big(R,2(1-\mu({\cal X}^0))\Big)\label{all3d}
\end{align}\end{subequations}
where $R\in[0,2]$, $k\in\{1,2,\ldots,r\}$ and $r$ is the number of ${\cal X}_k$ sets which is at most $|{\cal X}\setminus{\cal X}^0\cup{\cal X}_0|$.

The above discussion also holds for countable alphabet spaces $({\cal X},{\cal U})$. Next, we apply the above results to the minimax MCP defined by  \eqref{markov.dec.problem.amb}.

\section{Minimax Stochastic Control for Finite State and Control Spaces}\label{sec.minimaxstoch.control}
In this section, we investigate the infinite horizon minimax MCP  defined by \eqref{markov.dec.problem.amb} for finite state and control spaces. By employing the results of Section \ref{sec.minimaxstochastic.control}, we derive dynamic programming equation \eqref{n.DP7} and we introduce the corresponding policy iteration algorithm.

Consider the problem of minimizing the finite horizon version of \eqref{payofff} defined by 
\begin{equation}\label{an.probl.fin.hor}
J^*_{n}(x)=\inf_{u\in {\cal U}(x)}\sup_{  {{ Q}(\cdot| x, u)  \in \mathbf{B}_{R}(Q^o)(x,u)}}{\mathbb E}_x^g\Big\{\sum_{k=0}^{n-1}f(x_k,u_k)\Big\}.
\end{equation}
Let $V:{\cal X}\longmapsto\mathbb{R}$ denote the value function corresponding to \eqref{an.probl.fin.hor}. Then $V$ satisfies the dynamic programming equation \cite{2014arXiv1402.1009T,ctc2012}
\begin{subequations}
\begin{eqnarray}
V_n(x)&=& 0,\quad \forall x\in{\cal X}\\
V_j(x)&=&\inf_{u\in{\cal U}(x)}\sup_{Q(\cdot|x,u)\in \mathbf{B}_R(Q^o)(x,u)}\nonumber\\ [-1.5ex]\label{eq.inf.average.dp.init}\\[-1.5ex]
&&\Big\{f(x,u)+\sum_{z\in {\cal X}}V_{j+1}(z)Q(z|x,u)\Big\},\ j=0,1,\dots,n-1,\ x\in {\cal X}.\nonumber
\end{eqnarray}\end{subequations}
By applying \eqref{f2n}, with $\ell(\cdot)=V_{j+1}(\cdot)$ and $\mu(\cdot)=Q^o(\cdot|x,u)$, then \eqref{eq.inf.average.dp.init} is equivalent to the dynamic programming equation 
\begin{equation}\label{eq.inf.average.dp.equiv}
V_j(x)=\inf_{u\in{\cal U}(x)}\Big\{f(x,u){+}\sum_{ z\in {\cal X}}V_{j+1}(z)Q^o(z|x,u){+}\frac{R}{2}\Big(\sup_{z\in{\cal X}}V_{j+1}(z){-}\inf_{z\in{\cal X}}V_{j+1}(z)\Big)\Big\}.
\end{equation}
Moreover, by applying \eqref{all3} with $\nu^*(\cdot)=Q^*(\cdot|x,u)$, where $Q^*(\cdot|x,u)$, $(x,u)\in \mathbb{K}$ is the maximizing conditional distribution and $\mu(\cdot)=Q^o(\cdot|x,u)$, $(x,u)\in\mathbb{K}$, then \eqref{eq.inf.average.dp.init} is equivalent to
\begin{equation}
V_j(x)=\inf_{u\in{\cal U}(x)}\Big\{f(x,u){+}\sum_{z\in {\cal X}}V_{j+1}(z)Q^*(z|x,u)\Big\}.
\end{equation}

Define $\overline{V}_j(x)=V_{n-j}(x)$. Then from \eqref{eq.inf.average.dp.init}, $\overline{V}_j(\cdot)$ satisfies the equation 
\begin{eqnarray}
\overline{V}_j(x)&=&\inf_{u\in{\cal U}(x)}\sup_{Q(\cdot|x,u)\in \mathbf{B}_R(Q^o)(x,u)}\nonumber\\ [-1.5ex]\label{equiv.form.c}\\[-1.5ex]
&&\Big\{f(x,u)+\sum_{z\in {\cal X}}\overline{V}_{j-1}(z)Q(z|x,u)\Big\},\ j=0,1,\dots,n-1.\nonumber
\end{eqnarray}
We rewrite \eqref{equiv.form.c} as follows.
\begin{eqnarray}
&&\overline{V}_j(x)+\frac{1}{j}\overline{V}_j(x)\nonumber\\ [-1.5ex]\label{Vbar.eq1}\\[-1.5ex]
 &&=\inf_{u\in{\cal U}(x)}\sup_{Q(\cdot|x,u)\in \mathbf{B}_R(Q^o)(x,u)}\Big\{f(x,u)+\sum_{z\in {\cal X}}Q(z|x,u)\Big(\overline{V}_{j-1}(z)+\frac{1}{j}\overline{V}_j(x)\Big)\Big\}.\nonumber
\end{eqnarray}

Next, we introduce the following standard assumption \cite{varayia86}.
\begin{assumption}\label{assupt.1}
There exists a pair $(V(\cdot),J^*)$, $V:{\cal X}\longmapsto \mathbb{R}$ and $J^*\in\mathbb{R}$, such that
\begin{equation}\label{a.limitVtoJ}
\lim_{j\rightarrow\infty}\big(\overline{V}_j(x)-jJ^*\big)=V(x),\quad \forall x\in {\cal X}.
\end{equation}
\end{assumption}

Under Assumption \ref{assupt.1}, then \begin{equation}\label{limitVtoJ}
\lim_{j\rightarrow\infty}\frac{1}{j}\overline{V}_j(x)=J^*,\quad \forall x\in {\cal X}
\end{equation}
and the limit does not depend on $x\in {\cal X}$. In addition, by taking the supremum with respect to $x\in {\cal X}$ on both sides of \eqref{a.limitVtoJ}, by virtue of the the finite cardinality of ${\cal X}$, we can exchange the limit and the supremum to obtain
\begin{equation}\label{b.limitVtoJ}
\lim_{j\rightarrow\infty}\sup_{x\in {\cal X}}\big(\overline{V}_j(x)-jJ^*\big)=\sup_{x\in {\cal X}}\lim_{j\rightarrow\infty}\big(\overline{V}_j(x)-jJ^*\big)=\sup_{x\in {\cal X}}V(x).
\end{equation}

By Assumption \ref{assupt.1} and by \eqref{limitVtoJ} we have the following identities.
\begin{align*}
&J^*+V(x)\\
=&\lim_{j\rightarrow\infty}\Big(\frac{1}{j}\overline{V}_j(x)+(\overline{V}_j(x)-jJ^*)\Big)\\
\overset{(a)}=&\lim_{j\rightarrow\infty}\inf_{u\in{\cal U}(x)}
  \sup_{Q(\cdot|x,u)\in \mathbf{B}_R(Q^o)(x,u)}\Big\{f(x,u){+}\sum_{z\in {\cal X}}Q(z|x,u)\Big(\overline{V}_{j-1}(z)+\frac{1}{j}\overline{V}_j(x)\Big){-}jJ^*\Big\}\\
\overset{(b)}=&\lim_{j\rightarrow\infty}\inf_{u\in{\cal U}(x)}\Big\{f(x,u)-jJ^*+\sum_{z\in {\cal X}}Q^o(z|x,u)\Big(\overline{V}_{j{-}1}(z)+\frac{1}{j}\overline{V}_j(x)\Big)\\
& +\frac{R}{2}\Big(\sup_{z{\in}{\cal X}}\Big( \overline{V}_{j-1}(z)+\frac{1}{j}\overline{V}_j(x)\Big){-}\inf_{z{\in}{\cal X}}\Big(\overline{V}_{j{-}1}(z)+\frac{1}{j}\overline{V}_j(x)\Big)\Big)\Big\} \\
\overset{(c)}=&\lim_{j\rightarrow\infty}\inf_{u\in{\cal U}(x)}\Big\{f(x,u)+\sum_{z\in{\cal X}}Q^o(z|x,u)\big(\overline{V}_{j-1}(z)-(j-1)J^*+\frac{1}{j}\overline{V}_j(x){-}J^*\big)\\
& +\frac{R}{2}\Big(\sup_{z\in{\cal X}}\big(\overline{V}_{j-1}(z)-jJ^*\big)-\inf_{z\in{\cal X}}\big(\overline{V}_{j-1}(z)-jJ^*\big)\Big)\Big\} 
\end{align*}
where 

(a) is obtained by using \eqref{Vbar.eq1};

(b) is obtained by using the equivalent formulation \eqref{eq.inf.average.dp.equiv}; 

(c) is obtained by adding and subtracting $J^*(1+j\frac{R}{2})$. 

\noindent Since ${\cal U}$ and ${\cal X}$ are of finite cardinality we can interchange the limit and the minimization and maximization operations, to arrive to the following dynamic programming equation. 
\begin{equation}\label{dpequation.average.payoff}
J^*+V(x)=\min_{u\in{\cal U}(x)}\Big\{f(x,u)+\sum_{z\in {\cal X}}Q^o(z|x,u)V(z)+\frac{R}{2}\Big(\sup_{z\in {\cal X}}{V}(z)-\inf_{z\in {\cal X}}{V}(z)\Big)\Big\}.
\end{equation}
Clearly, by \eqref{f2n}, dynamic programming equation \eqref{dpequation.average.payoff} is equivalently expressed as follows.
\begin{equation}\label{equiv2.dpequation.average.payoff}
J^*+V(x)=\min_{u\in{\cal U}(x)}\max_{ {{ Q}(\cdot| x, u)  \in \mathbf{B}_{R}(Q^o)(x,u)}}\Big\{f(x,u)+\sum_{z\in {\cal X}}Q(z|x,u)V(z)\Big\}.
\end{equation}

Next, we state the first main Theorem of this section.
\begin{theorem}\label{DP.inf.ave.lemmac}
Suppose ${\cal X}$ and ${\cal U}$ are of finite cardinality and Assumption \ref{assupt.1} holds. If there exists a solution $(V,J^*)$ to the dynamic programming equation \eqref{dpequation.average.payoff}, and $g^*$ is a stationary policy such that $g^*(x)$ attains the minimum in the right-hand side of \eqref{dpequation.average.payoff} for every $x$, then $g^*$ is an optimal policy and $J^*$ is the minimum average cost.
\end{theorem}
\begin{proof}
 Let $g\in G$ be any policy and $u\in {\cal U}(x)$. Since $(V,J^*)$ satisfies the dynamic programming equation \eqref{dpequation.average.payoff}, which is equivalent to \eqref{equiv2.dpequation.average.payoff}, and by the definition of $g^*$ then
 \begin{eqnarray}\label{dyn.prog.eq.x439i}
 &&f(x,u)+\sum_{z\in {\cal X}}Q^o(z|x,u)V(z)+\frac{R}{2}\Big(\max_{z\in {\cal X}}{V}(z)-\min_{z\in {\cal X}}{V}(z)\Big)\\
 &=&\max_{Q(\cdot|x,u)\in \mathbf{B}_R(Q^o)(x,u)}\Big\{f(x,u)+\sum_{z\in {\cal X}}Q(z|x,u)V(z)\Big\}\nonumber\\
 &\geq &\max_{Q(\cdot|x,g^*(x))\in \mathbf{B}_R(Q^o)(x,g^*(x))}\Big\{f(x,g^*(x))\hspace{.1cm} + \sum_{z\in { \cal X}}Q(z|x,g^*(x))V(z)\Big\} \nonumber\\
 &=&J^*+V(x).\nonumber
 \end{eqnarray}
 Denoting the maximization with respect to $Q(\cdot|x,u)$ in \eqref{dyn.prog.eq.x439i} by $Q^*(\cdot|x,u)$ and the corresponding expectation by $\mathbb{E}^{g,Q^*}$, and taking expectation on both sides of \eqref{dyn.prog.eq.x439i}, we have
\begin{eqnarray}
\mathbb{E}^{g,Q^*}\Big(f(x_j,u_j)\Big)
&\geq&  J^*+\mathbb{E}^{g,Q^*}\Big(V(x_j)\Big)-\mathbb{E}^{g,Q^*}\Big(\sum_{z\in {\cal X}}Q^*(z|x_{j},u_{j})V(z)\Big)\nonumber\\ [-1.5ex]\label{dyn.prog.eq.439}\\[-1.5ex]
 &=&J^*+\mathbb{E}^{g,Q^*}\Big(V(x_j)\Big)-\mathbb{E}^{g,Q^*}\Big(V(x_{j+1})\Big).\nonumber
\end{eqnarray}
Then, from \eqref{payofff} we have that for all $g\in G$, 
\begin{align*}
J(\pi)&\geq \liminf_{j\rightarrow\infty}\Big(\frac{1}{j}\sum_{k=0}^{j-1}\mathbb{E}^{g,Q^*}\big(f(x_k,u_k)\big)\Big)\\
&\overset{(a)}\geq \liminf_{j\rightarrow\infty}\Big(J^*+\frac{1}{j}\Big(\mathbb{E}^{g,Q^*}\big(V(x_0)\big)-\mathbb{E}^{g,Q^*}\big(V(x_j)\big)\Big)\Big)\\
&\overset{(b)}=J^*
\end{align*}
where 

(a) is obtained by using \eqref{dyn.prog.eq.439}; 

(b) is obtained because the last term vanishes as $j\rightarrow\infty$. 

\noindent Thus, $J^*\leq \inf_{g\in G}J(g,x)$. However, when $g$ is replaced by $g^*$ equality holds throughout, and as a result $g^*$ is optimal, that is, $J^*=J^*(x)=\inf_{g\in G}J(g,x)$, $g^*\in G$ is an average cost optimal policy and $J^*$ is the value.
\end{proof}

%
%
%
%
\subsection{Existence}\label{sec.minimaxDP}
Dynamic programming equation \eqref{dpequation.average.payoff} and hence Theorem \ref{DP.inf.ave.lemmac}, are valid under Assumption \ref{assupt.1}. Here, we characterize the solution of the infinite horizon minimax average cost MCM,  under the standard irreducibility condition,  on the nominal transition probabilities of the controlled process. First, we introduce some notation.

Identify the state space ${\cal X}$ by ${\cal X}=\{x_1,x_2,\dots,x_{|{\cal X}|}\}$ consisting of ${|{\cal X}|}$ elements. Then, any function $V: {\cal X}\longmapsto \mathbb{R}$ may be represented by a vector in $\mathbb{R}^{|{\cal X}|}$, as follows.
\begin{align*}
  V=\left(\begin{array}{ccc}
               V(x_1) &\cdots &V(x_{|{\cal X}|})
             \end{array}
  \right)^T\in \mathbb{R}^{|{\cal X}|}.
\end{align*}
Any stationary control policy $g\in G_{SM}$, $g:{\cal X}\longmapsto \mathbb{R}$, may also be identified with a $g\in \mathbb{R}^{|{\cal X}|}$. For any $g$, let $Q(g)\in\mathbb{R}^{{|{\cal X}|}\times {|{\cal X}|}}$ defined by $Q(g)_{ij}=P(x_{t+1}=x_i|x_t=x_j,u_t=g(x_j))$ and
 \begin{align*}
  f(g)=\left(\begin{array}{ccc}
               f(x_1,g(x_1)) &\cdots &f(x_{|{\cal X}|},g(x_{|{\cal X}|}))
             \end{array}
  \right)^T\in \mathbb{R}^{|{\cal X}|}.
\end{align*}
Let $q_0\in \mathbb{R}^{|{\cal X}|}$ be defined by 
$q_0(x_i)\tri P(\{x_0=x_i\})$, $i=1,\dots,{|{\cal X}|}$ and $e\tri(1,\cdots,1)^T\in \mathbb{R}^{|{\cal X}|}$.

The maximization of the expected $n$-stage cost, for a fixed $q_0(x)\in \mathbb{R}^{|{\cal X}|}$, is given by\footnote{The notation $J_n(g,q_0)$ means that $q_0(x)$ is fixed instead of $x_0=x$.}
\begin{eqnarray}\label{finite.average.cost1}
J_n(g,q_0)\triangleq J_n(g,x)q_0^T(x)&=&\max_{  {{ Q}(\cdot| x, u)  \in \mathbf{B}_{R}(Q^o)(x,u)}}{\mathbb E}^g\Big\{\sum_{k=0}^{n-1}f(x_k,u_k)\Big\}\nonumber\\ 
&=&\max_{ {{ Q}(\cdot| x, u)  \in \mathbf{B}_{R}(Q^o)(x,u)}}\Big\{\sum_{k=0}^{n-1}q_0^TQ(g)^kf(g)\Big\}\\
&=&\max_{ {{ Q}(\cdot| x, u)  \in \mathbf{B}_{R}(Q^o)(x,u)}}q_0^T\Big\{\sum_{k=0}^{n-1}Q(g)^k\Big\}f(g).\nonumber
\end{eqnarray}
With $Q^*(\cdot|x,u)$ denoting the maximizing conditional distribution, then \eqref{finite.average.cost1} is equivalent
\begin{equation*}
\max_{  {{ Q}(\cdot| x, u)  \in \mathbf{B}_{R}(Q^o)(x,u)}}q_0^T\Big\{\sum_{k=0}^{n-1}Q(g)^k\Big\}f(g)
=q_0^T\Big\{\sum_{k=0}^{n-1}Q^*(g)^k\Big\}f(g).
\end{equation*}
Hence, the maximizing average cost per unit-time is given by 
\begin{eqnarray*}
J(g,q_0)
&=&\limsup_{n\rightarrow\infty}\max_{  {{ Q}(\cdot| x, u)  \in \mathbf{B}_{R}(Q^o)(x,u)}}\frac{1}{n}{\mathbb E}^g\Big\{\sum_{k=0}^{n-1}f(x_k,u_k)\Big\}\\
&=&\limsup_{n\rightarrow\infty}\frac{1}{n}q_0^T\Big\{\sum_{k=0}^{n-1}Q^*(g)^k\Big\}f(g).
\end{eqnarray*}
Since $q_0\in \mathbb{R}^{|{\cal X}|}$ and $f(g)\in \mathbb{R}^{|{\cal X}|}$ are independent of $n$, we only need to investigate the conditions under which the following limit exists
\begin{equation*}
\lim_{n\rightarrow\infty}\frac{1}{n}\sum_{k=0}^{n-1}Q^*(g)^k.
\end{equation*}
The next Lemma follows directly from \cite[Lemma 5.4]{varayia86}.
\begin{lemma}\label{Ces.lemma1}
If $Q^*\in \mathbb{R}^{|{\cal X}|\times |{\cal X}|}_+$ is a stochastic matrix, then the Cesaro limit
\begin{equation}
\lim_{n\rightarrow\infty}\frac{1}{n}\sum_{k=0}^{n-1}(Q^*)^k=Q^*_1
\end{equation} always exist. The matrix $Q^*_1\in\mathbb{R}_+^{|{\cal X}|\times |{\cal X}|}$ is a stochastic matrix and it is the solution of the equation
\begin{equation}\label{Cesaro.eq.b}
Q^*_1Q^*=Q^*_1.
\end{equation}
\end{lemma}

In view of Lemma \ref{Ces.lemma1}, the maximization of the average cost per unit-time of a stationary Markov control policy is given by
\begin{equation}\label{cost.initial}
J(g,q_0)=q_0^TQ^*_1(g)f(g)
\end{equation}
where $Q_1^*(g)$ and $Q^*(g)$ are related by \eqref{Cesaro.eq.b}. We recall the following definition of reducible stochastic matrix from \cite[page 44]{varayia86}.

\begin{definition}
A stochastic matrix $P\in \mathbb{R}_+^{|{\cal X}|\times |{\cal X}|}$ is said to be reducible if by row and column permutations it can be placed into block upper-triangular form
\begin{equation*}
P=\left(
                    \begin{array}{cc}
                      P_1 & P_2  \\
                      0 & P_3\\
                    \end{array}
                  \right), \quad \mbox{where $P_1$, $P_2$ are square matrices.}
                  \end{equation*}
                  A stochastic matrix which is not reducible is said to be irreducible.
\end{definition}

Next, we recall the following Lemma from \cite[Lemma 5.7]{varayia86}.
\begin{lemma}\label{lemma.irreducability}
Let $Q^*\in\mathbb{R}_+^{|{\cal X}|\times |{\cal X}|}$ be an irreducible stochastic matrix. Then, there exists a unique vector $q$ such that
\begin{equation*}
Q^*q=q,\quad e^T q=1, \quad q(x_i)>0 \ \ \mbox{for all} \ \  {x_i\in \cal X}.
\end{equation*} Moreover, the matrix $Q_1^*$ associated with $Q^*$ in \eqref{Cesaro.eq.b} has all rows equal to $q$.
\end{lemma}

Note that, \eqref{cost.initial} depends on the probability distribution $q_0$ of the initial state. However, if $Q^*_1$ is assumed to be an irreducible stochastic matrix, by Lemma \ref{lemma.irreducability} 
\begin{equation}\label{cost.independ.initial}
J(g,q_0)=q_0^TQ_1^*(g)f(g)=q(g)^Tf(g)\equiv J(g)
\end{equation}
where $q(g)$ is the unique invariant probability distribution, that is, $Q^*(g)q(g)=q(g)$, and the average cost per unit-time $J(g,q_0)\equiv J(g)$ is independent of the initial distribution. Hence, for the remainder of this section, we will assume that for every stationary Markov control policy $g\in G_{SM}$, the stochastic matrix $Q^*(g)$ is irreducible. The next proposition summarizes the above results.

\begin{proposition}\label{prop1.dp.inf.hor} \cite{vanSchuppen10}
Let $g\in G_{SM}$ be a stationary Markov control policy, $g:\cal X\longmapsto {\cal U}$ and assume that $Q^*(g)\in\mathbb{R}_+^{|{\cal X}|\times |{\cal X}|}$ is irreducible. 

\noindent Then the following hold.
\ben
\item[(a)] There exists a unique $q(g)\in \mathbb{R}^{|{\cal X}|}_+$ such that
\begin{equation}
Q^*(g)q(g)=q(g), \quad e^Tq=1.
\end{equation}
\item[(b)] The average cost per unit-time associated with the control policy $g\in G_{SM}$ is \begin{equation}
J(g)=q(g)^Tf(g).
\end{equation}
\item[(c)] There exists a $V(g)\in\mathbb{R}^{|{\cal X}|}$ such that
\begin{equation}
J(g)e+V(g)=f(g)+Q^*(g)V(g).
\end{equation}
\een
\end{proposition}
\begin{proof}
Part (a) and (b) follows from Lemma \eqref{lemma.irreducability} and the discussion above it. For part (c) see \cite{vanSchuppen10}.
\end{proof}

\begin{lemma}\label{lemma1.dp.inf.hor}
Assume the following hold.
\ben
\item For any stationary control policy $g\in G_{SM}$, $Q^*(g)\in\mathbb{R}_+^{|{\cal X}|\times |{\cal X}|}$ is irreducible.
\item There exists a $g^*\in G_{SM}$ such that \begin{equation*}
J^*=\inf_{g\in G_{SM}}J(g).
\end{equation*}
\een
Then there exists an $(V(g^*,\cdot),J^*)$, $V(g^*,\cdot):{\cal X}\longmapsto \mathbb{R}$ and $J^*\in \mathbb{R}$ which is a solution to the dynamic programming equation 
\begin{equation*}
J^*+V(g^*,x)=\min_{u\in{\cal U}}\Big\{f(x,u)+\sum_{z\in{\cal X}}Q^*(z|x,u)V(g^*,z)\Big\}.
\end{equation*}
\end{lemma}
\begin{proof}
By Proposition \ref{prop1.dp.inf.hor} (c), there exists a $V(g^*,\cdot):{\cal X}\longmapsto \mathbb{R}$ and $J^*$ such that for all $x\in{\cal X}$
\begin{equation}\label{appendC.DP.equation.a}
J^*+V(g^*,x)=f(x,g^*(x))+\sum_{z\in {\cal X}}Q^*(z|x,g^*(x))V(g^*,z).
\end{equation}
Then, for all $x\in{\cal X}$
\begin{equation*}
J^*+V(g^*,x)\geq \min_{u\in {\cal U}}\Big\{f(x,u)+\sum_{z\in {\cal X}}Q^*(z|x,u)V(g^*,z)\Big\}.
\end{equation*}
Define $g_1:{\cal X}\longmapsto {\cal U}$ as
\begin{equation*}
g_1(x)=\argmin_{u\in{\cal U}}\Big\{f(x,u)+\sum_{z\in{\cal X}}Q^*(z|x,u)V(g^*,z)\Big\}.
\end{equation*}
Suppose that for some $x_2\in{\cal X}$ strict inequality holds in \eqref{appendC.DP.equation.a}, then
\begin{equation}\label{appendC.DP.equation.b}
J^*+V(g^*,x)
> \min_{u\in {\cal U}}\Big\{f(x_2,u)+\sum_{z\in {\cal X}}Q^*(z|x_2,u)V(g^*,z)\Big\}.
\end{equation}
Multiplying \eqref{appendC.DP.equation.b} by $q(g_1)(x_0)>0$ and summing over $x_0\in{\cal X}$ yields
\begin{eqnarray*}
&&J^*+\sum_{x_0\in{\cal X}}q(g_1)(x_0)V(g^*,x_0)\\
&>&\min_{u\in{\cal U}}\Big\{\sum_{x_0\in\cal X}q(g_1)(x_0)f(x_0,u)+\sum_{x_0\in\cal X}q(g_1)(x_0)\sum_{z\in\cal X}Q^*(z|x_0,u)V(g^*,z)\Big\}\\
&=&\sum_{x_0\in\cal X}q(g_1)(x_0)f(x_0,g_1(x_0))+\sum_{x_0\in\cal X}q(g_1)(x_0)\sum_{z\in\cal X}Q^*(z|x_0,g_1(x_0))V(g^*,z)\\
&=&J(g_1)+\sum_{z\in\cal X}q(g_1)V(g^*,z), \quad \mbox{by Proposition \ref{prop1.dp.inf.hor} (a)}
\end{eqnarray*}
which gives $J^*>J(g_1)$, contradicting assumption 2. Hence, equality holds in \eqref{appendC.DP.equation.a}, for every $x\in\cal X$.
\end{proof}

Next, we state the second main Theorem of this section.

\begin{theorem}\label{DP.inf.hor.Theorem.average.lim}
Assume that for all stationary Markov control policies $g\in G_{SM}$, and for a given total variation parameter $R\in[0,R_{\max}]\subset[0,2]$, the maximizing transition matrix $Q^*(g)$ is irreducible. Then the following hold.
\ben
\item[(a)] There exists a solution $(V,J^*)$, $V:\cal X\longmapsto \mathbb{R}$, $J^*\in \mathbb{R}$ to the dynamic programming equation 
\begin{equation}\label{theorem.dp.eq}
J^*+V(x)
=\min_{u\in\cal U}\Big\{f(x,u)+\sum_{z\in\cal X}Q^*(z|x,u)V(z)\Big\}
\end{equation}
or, to the equivalent dynamic programming equation
\begin{equation}\label{theorem.dp.eq.1a}
J^*+V(x)=\min_{u\in\cal U}\Big\{f(x,u)
+\sum_{z\in\cal X}Q^o(z|x,u)V(z)
+\frac{R}{2}\Big(\max_{z\in\cal X}V(z)-\min_{z\in\cal X}V(z)\Big)\Big\}
\end{equation}
where $\max_{z\in\cal X}V(z)$ denotes component-wise maximum and similarly for the minimum. The maximizing conditional distribution $Q^*(\cdot|x,u)$, $(x,u)\in\mathbb{K}$ is given by \eqref{all3}, where $\nu^*(\cdot)$, $\mu(\cdot)$ and $\ell(\cdot)$ are replaced by $Q^*(\cdot|x,u)$, $Q^o(\cdot|x,u)$ and $V(\cdot)$, respectively, i.e.,
\begin{subequations}\label{Q.matrix.limited.a}
\begin{align}
\label{Q.matrix.limited.a1}&Q^*({\cal X}^+|x,u)=Q^o({\cal X}^+|x,u)+\frac{\alpha}{2}\\
\label{Q.matrix.limited.a2}&Q^*({\cal X}^-|x,u)=\Big(Q^o({\cal X}^-|x,u)-\frac{\alpha}{2}\Big)^+\\
\label{Q.matrix.limited.a3}&Q^*({\cal X}_k|x,u)= \Big(Q^o({\cal X}_k|x,u)-\Big(\frac{\alpha}{2}-\sum_{j=1}^kQ^o({\cal X}_{k-1}|x,u)\Big)^+\Big)^+\\
\label{Q.matrix.limited.a4}&\alpha = \min\Big(R,2(1-Q^o({\cal X}^+|x,u))\Big)
\end{align}
\end{subequations}
and
\begin{subequations}\label{Q.matrix.limited.b}
\begin{align}
\label{Q.matrix.limited.b1}&{\cal X}^+\triangleq  \Big\{x\in{\cal X}:V(x)=\max\{V(x):x\in{\cal X}\}\Big\}\\
\label{Q.matrix.limited.b2}&{\cal X}^-\triangleq  \Big\{x\in{\cal X}:V(x)=\min\{V(x):x\in{\cal X}\}\Big\}\\
\label{Q.matrix.limited.b3}&{\cal X}_k \triangleq \Big\{x\in {\cal X}:V(x)=\min\big\{V(\alpha): \alpha \in {\cal X}\setminus {\cal X}^0\cup(\bigcup_{j=1}^k{\cal X}_{j-1})\big\} \Big\}
\end{align}\end{subequations}
where $k=1,2,\dots,r$ (see Section \ref{subsec.finite.extr}).
\item[(b)] If $g^*(x)$ attains the minimum in \eqref{theorem.dp.eq} or equivalently in \eqref{theorem.dp.eq.1a} for every $x$, then $g^*$ is an average cost optimal policy.
\item[(c)] The minimum average cost is $J^*$.
\een
\end{theorem}
\begin{proof}
Theorem \ref{DP.inf.hor.Theorem.average.lim} is obtained by combining Theorem \ref{DP.inf.ave.lemmac} and Lemma \ref{lemma1.dp.inf.hor} and by applying the results of Section \ref{sec.minimaxstochastic.control}.
\end{proof}

The main observation is that in specific applications one may employ either dynamic programming equation \eqref{theorem.dp.eq} or \eqref{theorem.dp.eq.1a}.

\subsubsection{Policy Iteration Algorithm}
\label{subsec.inf_ave_policy_iteration_alg}
In this section, we provide a modified version of the classical policy iteration algorithm for average cost dynamic programming \cite{varayia86,vanSchuppen10}. From part (a) of Theorem \ref{DP.inf.hor.Theorem.average.lim}, the policy evaluation and policy improvement steps of a policy iteration algorithm must be performed using the maximizing conditional distribution obtained under total variation distance ambiguity constraint. Moreover, one needs to guarantee that for the given total variation parameter $R$, the corresponding maximizing matrix $Q^*$ is irreducible, otherwise, Algorithm \ref{alg.pol.iter.aver.cost} may not be sufficient to give the optimal policy and the minimum cost. In general, $R\in[0,R_{\max}]\subseteq[0,2]$, and $R_{\max}$ is strictly less than $2$. This generality will be discussed in Section \ref{subsec.inf_ave_gener_DP} for general Borel spaces.\\

\begin{algorithm}
(Policy iteration)
\ben 
\item Let $m=0$ and select an arbitrary stationary Markov control policy $g_0:\cal X\longmapsto {\cal U}$.
\item (Policy Evaluation) Solve the equation 
\begin{equation}\label{pol.eval.V}
J_{Q^o}(g_m)e{+}V_{Q^o}(g_m){=}f(g_m){+}Q^o(g_m)V_{Q^o}(g_m)
\end{equation}
for $J_{Q^o}(g_m)\in\mathbb{R}$ and $V_{Q^o}(g_m)\in\mathbb{R}^{|{\cal X}|}$. Identify the support sets of \eqref{pol.eval.V} using \eqref{Q.matrix.limited.b}, and construct the matrix $Q^*(g_m)$ using \eqref{Q.matrix.limited.a}. Solve the equation
\begin{equation}\label{pol.eval.V.TV}
J_{Q^*}(g_m)e{+}V_{Q^*}(g_m){=}f(g_m){+}Q^*(g_m)V_{Q^*}(g_m)
\end{equation} for $J_{Q^*}(g_m)\in\mathbb{R}$ and $V_{Q^*}(g_m)\in\mathbb{R}^{|{\cal X}|}$.
\item (Policy Improvement) Let \begin{equation}
g_{m+1}=\argmin_{g\in \mathbb{R}^{|{\cal X}|}}\Big\{f(g)+Q^*(g)V_{Q^*}(g_m)\Big\}.
\end{equation}
\item If $g_{m+1}=g_m$, let $g^*=g_m$; else let $m=m+1$ and return to step 2.\\
\een
\label{alg.pol.iter.aver.cost}
\end{algorithm}

In Section \ref{working example3}, we illustrate how policy iteration algorithm for infinite horizon average cost dynamic programming is implemented through an example. 

\subsubsection{Limitations}
\label{subsec.inf_ave_limitation}
Part (a) of Theorem \ref{DP.inf.hor.Theorem.average.lim}, indicates that for a stationary Markov control policy $g\in G_{SM}$, and for an irreducible stochastic matrix $Q^*$ there exists a solution to the dynamic programming equation \eqref{theorem.dp.eq}. Moreover, the maximizing stochastic matrix $Q^*$ which is given by \eqref{Q.matrix.limited.a}, is calculated based on the support sets \eqref{Q.matrix.limited.b}, the nominal stochastic matrix $Q^o$, and the value of the total variation parameter $R\in [0,R_{\max}]$. Hence, in order to apply policy iteration algorithm for average-cost dynamic programming one needs to know in advance that, for a given total variation parameter $R\in[0,2]$, and an irreducible nominal stochastic matrix $Q^o$, the maximizing stochastic matrix $Q^*$ is also irreducible. Otherwise, policy iteration algorithm may not be sufficient to give the optimal policy and the minimum cost. In particular, as we show next, if the irreducibility condition is not satisfied then the policy iteration algorithm need not have a unique solution.

As an example (inspired by \cite{puterman94}), consider the stochastic control system shown in Fig.\ref{fig.exd.11a}, with state-space ${\cal X}=\{1,2,3\}$ and control set ${\cal U}=\{u_1,u_2\}$.
\begin{figure}[!htbp]
        \centering
                \includegraphics[height=5.5cm,width=.55\linewidth]{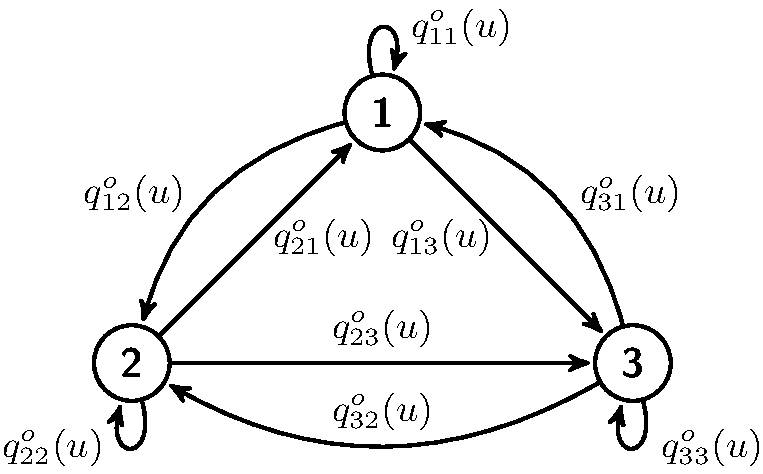}      
        \caption[]{Nominal Stochastic Control System.}\label{fig.exd.11a}
\end{figure}
Let the nominal transition probability under controls $u_1$ and $u_2$ to be given by
\begin{equation}\label{matrices.red}
  Q^o(u_1)=\frac{1}{9}\left(
                    \begin{array}{ccc}
                      0 & 5 & 4 \\
                      0 & 9 & 0 \\
                      0 & 0 & 9 \\
                    \end{array}
                  \right),\quad Q^o(u_2)=\frac{1}{9}\left(
                    \begin{array}{ccc}
                      2 & 7 & 0 \\
                      3 & 6 & 0 \\
                      8 & 0 & 1 \\
                    \end{array}
                  \right).
\end{equation}
The cost function under each state and action is given by
\begin{equation*}
f(1,u_1)=2,\ f(2,u_1)=1,\ f(3,u_1)=3,\ f(1,u_2)=0.5,\ f(2,u_2)=3,\ f(3,u_2)=0.
\end{equation*}
Clearly, from \eqref{matrices.red}, this control system the nominal transition probability matrix, under both controls, is reducible, since the system under controls $u_1$ and $u_2$ contains more than one communication class\footnote{States $i$ and $j$ belong to the same communication class if and only if each of these states can reach and be reached by the other.}. Using policy iteration Algorithm \ref{alg.pol.iter.aver.cost} with initial policies $g_0(1)=g_0(2)=g_0(3)=u_1$, the optimality equation \eqref{pol.eval.V} for this system may be written as 
\begin{equation*}
J_{Q^o}\left(\begin{array}{c}
     1 \\
     1 \\
     1
   \end{array}\right){+}V_{Q^o}(g_0){=}\left(\begin{array}{c}
     2 \\
     1 \\
     3
   \end{array}\right){+}\frac{1}{9}\left(
                              \begin{array}{ccc}
                                0 & 5 & 4 \\
                                0 & 9 & 0 \\
                                0 & 0 & 9 \\
                              \end{array}
                            \right)V_{Q^o}(g_0)
\end{equation*}
and hence
\begin{eqnarray*}
J_{Q^o}+V_{Q^o}(g_0,1)&=&2+\frac{5}{9}V_{Q^o}(g_0,2)+\frac{4}{9}V_{Q^o}(g_0,3)\\
J_{Q^o}+V_{Q^o}(g_0,2)&=&1+V_{Q^o}(g_0,2)\Longrightarrow J_{Q^o}=1\\
J_{Q^o}+V_{Q^o}(g_0,3)&=&3+V_{Q^o}(g_0,3)\Longrightarrow J_{Q^o}=3.
\end{eqnarray*}
The second and third equations show that the system is inconsistent, and hence, the policy iteration algorithm fails to give the optimal policy and the minimum cost.

Moreover, even if $Q^o$ is an irreducible stochastic matrix, as the value of total variation parameter $R$ increases the maximizing stochastic matrix $Q^*(R)$, eventually, will be transformed into a reducible stochastic matrix. Hence, our proposed method for solving minimax stochastic control problem with average cost is valid only for a specific range of values of total variation parameter, in $R\in [0,R_{\max}]\subseteq [0,2]$. In particular, if $Q^o$ is an irreducible stochastic matrix then, for any given partition of the state-space, there exists an $R_{\max}\in[0,2)$ for which we distinguish the following two cases: \ben \item[(a)] for $0\leq R<R_{\max}$, $Q^*$ is an irreducible stochastic matrix. Theorem \ref{DP.inf.hor.Theorem.average.lim} is valid and policy iteration algorithm gives the optimal policy and the minimum cost.
  \item[(b)] for $R\geq R_{\max}$, $Q^*$ is a reducible stochastic matrix. Theorem \ref{DP.inf.hor.Theorem.average.lim} is not valid and policy iteration algorithm need not have a solution.
  \een

\begin{remark}
Consider $R\geq R_{\max}$. Then, an extended solution through a reduced dimensional state-space may be obtained as follows.  Due to the water-filling behavior of maximizing conditional distribution \eqref{Q.matrix.limited.a}, columns of $Q^*$ which correspond to states belonging to ${\cal X}\setminus {\cal X}^0$, become columns with all zero's, as total variation parameter $R$ increases. Whenever an all zero column appears, one can remove the corresponding state of that column, and hence $Q^*$ will be transformed back into an irreducible stochastic matrix of reduced order.
\end{remark}

\section{Minimax Stochastic Control for Borel Spaces}
\label{subsec.inf_ave_gener_DP}
In this section, we derive the general dynamic programming equation for Borel spaces $({\cal X},{\cal U})$ which solves the MDP for all values of $R\in[0,2]$. In addition, we derive a generalized policy iteration algorithm corresponding to the generalized dynamic programming equations  when the state and control spaces are of finite dimension. Note that, throughout this section we again suppose that Assumption \ref{assump.cost} holds.

\subsection{General Dynamic Programming} Throughout this section it is assumed that Assumptions~\ref{assump.cost} hold.
The characterization of optimal policies for the minimax MCP defined by \eqref{markov.dec.problem.amb}, will be based on the concept of a canonical triplet adopted to the current formulation (see \cite{hern1996discrete}).

Consider the MCM \eqref{inf.hor.mcm.def}, where $({\cal X},{\cal U})$ are Borel spaces, and let $h:{\cal X}\longmapsto \mathbb{R}$ be a bounded, continuous and non-negative function. Denote the expected $n$-stage cost, with a terminal cost $h$, policy $g$, and $x_0=x$, by $J_0(g,Q,x,h)=h(x)$,  and for $n\geq 1$, by
\begin{equation}
J_n(g,Q,x,h)=\mathbb{E}^g_x\Big\{\sum_{k=0}^{n-1}f(x_k,u_k)+h(x_n)\Big\}=J_n(g,Q,x)+\mathbb{E}^g_x\Big\{h(x_n)\Big\}
\end{equation} with $J_n(g,Q,x)=J_n(g,Q,x,0)$. The corresponding maximizing expected $n$-stage cost is given by
\begin{eqnarray}\label{gen.finite.cost}
J_n(g,x,h)&=&\sup_{Q(\cdot|x,u)\in \mathbf{B}_R(Q^o)(x,u)}\mathbb{E}^g_x\Big\{\sum_{k=0}^{n-1}f(x_k,u_k)+h(x_n)\Big\}\\
&=&\mathbb{E}^{g,Q^*}_x\Big\{\sum_{k=0}^{n-1}f(x_k,u_k)+h(x_n)\Big\}=J_n(g,x)+\mathbb{E}^{g,Q^*}_x\Big\{h(x_n)\Big\}\nonumber
\end{eqnarray}
with $J_n(g,x)=J_n(g,x,0)$, where $Q^*(\cdot|x,u)$ is the maximizing distribution. Then, 
\begin{equation}\label{gen.infinite.cost}
J_n^*(x,h)=\inf_{g\in G}J_n(g,x,h);\quad J_n^*(x)=\inf_{g\in G}J_n(g,x,h), \quad \mbox{if} \ h(\cdot)=0.
\end{equation}

Throughout this section it is assumed that there exists a policy $g\in G$ and an initial state $x\in {\cal X}$ such that $J(g,x)<\infty$ (i.e., see \eqref{payofff}).
The definition of a canonical triplet is intriduced next, following  \cite{Yushkevich1973,hern1996discrete} with a slight variation, to account for the extra terms, which enter the dynamic programming equation.
\begin{definition}\label{dfn.canonical}
Let $\rho$ and $h$ be real-valued, bounded, continuous, non-negative, measurable functions on $\cal X$ and $\varphi\in \mathbb{F}$ a given selector. Then $(\rho,h,\varphi)$ is said to be a canonical triplet if
\begin{equation}\label{canonical.triplet.eq}
J_n(g^{\infty},x,h)=J_n^*(x,h)=n\rho(x)+h(x),\quad \forall x\in {\cal X},\ n=0,1,\dots.
\end{equation}
A selector $\varphi\in \mathbb{F}$ (of a stationary policy $g^{\infty}\in G_{SM}$) is called canonical if it is an element of some canonical triplet. 
\end{definition}

Note that with the appropriate choice of $h$ as the terminal cost the policy $g^{\infty}$ is optimal for the $n$-stage problem for all $n=0,1,\dots$. The following Theorem characterizes the canonical triplets for the minimax problem, with respect to the new dynamic programming equation.

\begin{theorem}\label{theorme}
Suppose the supremum and infimum of $h(\cdot)$ and  $\rho(\cdot)$ over ${\cal X}$ is non-empty. Then $(\rho,h,\varphi)$ is a canonical triplet if and only if, for every $x\in\cal X$, the following hold.
\begin{alignat}{2}
&\tag{a} &&\hspace{-.6cm} \rho(x)=\inf_{u\in {\cal U}(x)}\Big\{\int_{{\cal X}}\rho(z)Q^o(dz|x,u)+\frac{R}{2}\Big(\sup_{z\in{\cal X}}\rho(z)-\inf_{z\in{\cal X}}\rho(z)\Big)\Big\}  \label{can.tripl.a}\\
&\tag{b} &&\hspace{-.6cm} \rho(x)+h(x)=\inf_{u\in {\cal U}(x)}\Big\{f(x,u){+}\int_{ {\cal X}} h(z)Q^o(dz|x,u)+\frac{R}{2}\Big(\sup_{z\in{\cal X}}h(z)-\inf_{z\in{\cal X}}h(z)\Big)\Big\} \label{can.tripl.b}\\
&\tag{c} &&\hspace{-.6cm} \varphi(x)\in {\cal U}(x) \ \mbox{attains the minimum in both \eqref{can.tripl.a} and \eqref{can.tripl.b}, that is,} \label{can.tripl.c}\\
& &&  \quad \rho(x)=\int_{{\cal X}}\rho(z)Q^o(dz|x,\varphi)+\frac{R}{2}\Big(\sup_{z\in{\cal X}}\rho(z)-\inf_{z\in{\cal X}}\rho(z)\Big)\label{can.tripl.c1}\\
& && \quad \rho(x)+h(x)=f(x,\varphi)+\int_{ {\cal X}}h(z)Q^o(dz|x,\varphi)+\frac{R}{2}\Big(\sup_{z\in{\cal X}}h(z)-\inf_{z\in{\cal X}}h(z)\Big)\label{can.tripl.c2}
\end{alignat}
or, equivalently, $(\rho,h,\varphi)$ is a canonical triplet if and only if for every $x\in\cal X$ the following hold.
\begin{alignat}{2}
&\tag{a'} &&\hspace{-.7cm} \rho(x)=\inf_{u\in {\cal U}(x)}\sup_{Q(\cdot|x,u)\in \mathbf{B}_R(Q^o)(x,u)}\int_{{\cal X}}\rho(z)Q(dz|x,u) \label{can.tripl.equiv.a}\\
&\tag{b'} &&\hspace{-.7cm} \rho(x)+h(x)=\inf_{u\in {\cal U}(x)}\sup_{Q(\cdot|x,u)\in \mathbf{B}_R(Q^o)(x,u)}\Big\{f(x,u){+}\int_{ \mathclap{\ {\cal X}}}\ h(z)Q(dz|x,u)\Big\}\label{can.tripl.equiv.b}\\
&\tag{c'} &&\hspace{-.7cm} \varphi(x)\in {\cal U}(x) \ \mbox{attains the minimum in \eqref{can.tripl.equiv.a} and \eqref{can.tripl.equiv.b}, that is,} \label{can.tripl.equiv.c}\\
& &&\quad \rho(x)=\sup_{Q(\cdot|x,u)\in \mathbf{B}_R(Q^o)(x,u)}\int_{{\cal X}}\rho(z)Q(dz|x,\varphi)\label{can.tripl.equiv.c1}\\
& && \quad \rho(x)+h(x)=\sup_{Q(\cdot|x,u)\in \mathbf{B}_R(Q^o)(x,u)}\Big\{f(x,\varphi)+\int_{ {\cal X}} h(z)Q(dz|x,\varphi)\Big\}\label{can.tripl.equiv.c2}
\end{alignat} 
\end{theorem}

Note that, if $(\rho,h,\varphi)$ is a canonical triplet, then so is $(\rho,h+N,\varphi)$ for any constant $N$. Next we proceed with the proof of Theorem \ref{theorme}.

%

\begin{proof}
(Necessity). Suppose that $(\rho,h,\varphi)$ is a canonical triplet, i.e., \eqref{canonical.triplet.eq} holds $\forall x\in\cal X$ and $n\geq 0$. From the analog of dynamic programming equation \eqref{eq.inf.average.dp.equiv} of Borel spaces, we have that
\begin{equation}\label{dp1.back.can.tripl}
V_j(x)=\inf_{u\in{\cal U}(x)}\Big\{f(x,u)+\int_{{\cal X}}V_{j+1}(z)Q^o(dz|x,u)+\frac{R}{2}\Big(\sup_{z\in {\cal X}}V_{j+1}(z)-\inf_{z\in {\cal X}}V_{j+1}(z)\Big)\Big\}.
\end{equation}
Define $\overline{V}_j(x)=V_{n-j}(x)$, ($j=0,1,\dots,n$). Then \eqref{dp1.back.can.tripl} may be written in the ``forward" form
\begin{equation}\label{dp1.forward.can.tripl}
\overline{V}_{j+1}(x)=\inf_{u\in{\cal U}(x)}\Big\{f(x,u)+\int_{{\cal X}}\overline{V}_{j}(z)Q^o(dz|x,u)+\frac{R}{2}\Big(\sup_{z\in {\cal X}}\overline{V}_{j}(z)-\inf_{z\in {\cal X}}\overline{V}_{j}(z)\Big)\Big\}.
\end{equation}
Substituting \eqref{dp1.forward.can.tripl} to \eqref{gen.finite.cost}-\eqref{gen.infinite.cost}, we have
\begin{multline}\label{proof.eq.a}
J_{n+1}^*(x,h)
=\inf_{u\in{\cal U}(x)}\Big\{f(x,u)+\int_{{\cal X}}J_{n}^*(z,h)Q^o(dz|x,u)\\
+\frac{R}{2}\Big(\sup_{z\in {\cal X}}J_{n}^*(z,h)-\inf_{z\in {\cal X}}J_{n}^*(z,h)\Big)\Big\}.
\end{multline}
Thus, from \eqref{canonical.triplet.eq} we have
\begin{multline}\label{star1}
(n+1)\rho(x)+h(x)=\inf_{u\in{\cal U}(x)}\Big\{ f(x,u)+\int_{ \cal X}\big(n\rho(z)+h(z)\big)Q^0(dz|x,u)\\
+\frac{R}{2}\Big(\sup_{z\in {\cal X}}\big(n\rho(z)+h(z)\big)-\inf_{z\in {\cal X}}\big(n\rho(z)+h(z)\big)\Big)\Big\}.
\end{multline}
Evaluating \eqref{star1} at $n=0$ we obtain \eqref{can.tripl.b}. Furthermore, since $\rho(\cdot)$, $h(\cdot)$ and $f(\cdot,\cdot)$ are bounded, then multiplying both sides of \eqref{star1} by $1/n$ and letting $n\longrightarrow \infty$ yields \eqref{can.tripl.a}.

Finally, for any deterministic stationary policy $g^{\infty}\in G_{SM}$, we have that
 \begin{multline}\label{star2}
J_{n+1}(g^{\infty},x,h)
=f(x,\varphi)+\int_{{\cal X}}J_{n}(g^{\infty},z,h)Q^o(dz|x,\varphi)\\
+\frac{R}{2}\Big(\sup_{z\in {\cal X}}J_{n}(g^{\infty},z,h)-\inf_{z\in {\cal X}}J_{n}(g^{\infty},z,h)\Big),\quad x\in {\cal X}.
\end{multline}
 Thus, if $\varphi\in \mathbb{F}$ satisfies \eqref{canonical.triplet.eq}, then by \eqref{proof.eq.a}-\eqref{star2} we have that
\begin{multline*}
(n+1)\rho(x)+h(x)
=f(x,\varphi)+\int_{{\cal X}}\big(n\rho(z)+h(z)\big)Q^o(dz|x,\varphi)\\
+\frac{R}{2}\Big(\sup_{z\in {\cal X}}\big(n\rho(z)+h(z)\big)-\inf_{z\in {\cal X}}\big(n\rho(z)+h(z)\big)\Big)
\end{multline*}
which, as before, gives \eqref{can.tripl.c1} and \eqref{can.tripl.c2}.

(Sufficiency). Conversely, suppose $(\rho,h,\varphi)$ satisfy \eqref{can.tripl.a}-\eqref{can.tripl.c}. 
Proceeding by induction equation \eqref{canonical.triplet.eq} is trivially satisfied when $n=0$. Suppose that is true for some $n\geq 0$.  Then,  the following is obtained
\begin{eqnarray*}
J_{n+1}^*(x,h)&=&\inf_{u\in{\cal U}(x)}\Big\{f(x,u)+\int_{ \cal X}\big(n\rho(z)+h(z)\big)Q^o(dz|x,u)\\
&&+\frac{R}{2}\Big(\sup_{z\in {\cal X}}\big(n\rho(z)+h(z)\big)-\inf_{z\in {\cal X}}\big(n\rho(z)+h(z)\big)\Big)\Big\}\\
&=&\inf_{u\in{\cal U}(x)}\Big\{f(x,u)+\int_{ \cal X}\big(n\rho(z)+h(z)\big)Q^*(dz|x,u)\Big\}\\
&\geq & \inf_{u\in{\cal U}(x)}\Big\{f(x,u)+\int_{ \cal X}h(z)Q^*(dz|x,u)\Big\}+n\inf_{u\in{\cal U}(x)}\Big\{\int_{ \cal X}\rho(z)Q^*(dz|x,u)\Big\}\\
&=& \inf_{u\in{\cal U}(x)}\Big\{f(x,u){+}\int_{ \cal X}h(z)Q^o(dz|x,u){+}\frac{R}{2}\Big(\sup_{z\in {\cal X}}h(z){-}\inf_{z\in {\cal X}}h(z)\Big)\Big\}\\
&&+n\inf_{u\in{\cal U}(x)}\Big\{\int_{ \cal X}\rho(z)Q^o(dz|x,u)+\frac{R}{2}\Big(\sup_{z\in {\cal X}}\rho(z)-\inf_{z\in {\cal X}}\rho(z)\Big)\Big\}\\
&=&(n+1)\rho(x)+h(x).
\end{eqnarray*}
On the other hand,
\begin{eqnarray*}
J_{n+1}^*(x,h)&\leq &J_{n+1}(g^{\infty},x,h)\\
&=&f(x,\varphi)+\int_{ \cal X}\big(n\rho(z)+h(z)\big)Q^o(dz|x,\varphi)\\
&&+\frac{R}{2}\Big(\sup_{z\in {\cal X}}\big(n\rho(z)+h(z)\big)-\inf_{z\in {\cal X}}\big(n\rho(z)+h(z)\big)\Big)\\
&=&f(x,\varphi)+\int_{ \cal X}\big(n\rho(z)+h(z)\big)Q^*(dz|x,\varphi)    \\
&=& f(x,\varphi)+\int_{ \cal X}h(z)Q^*(dz|x,\varphi)\Big\}+n\int_{ \cal X}\rho(z)Q^*(dz|x,\varphi)\\
&=&f(x,\varphi){+}\int_{ \cal X}h(z)Q^o(dz|x,\varphi){+}\frac{R}{2}\Big(\sup_{z\in {\cal X}}h(z){-}\inf_{z\in {\cal X}}h(z)\Big)\\
&&+n\Big\{\int_{ \cal X}\rho(z)Q^o(dz|x,\varphi)+\frac{R}{2}\Big(\sup_{z\in {\cal X}}\rho(z)-\inf_{z\in {\cal X}}\rho(z)\Big)\Big\}\\
&=&(n+1)\rho(x)+h(x)
\end{eqnarray*}
where the second and third equalities follow by applying (\ref{f2n}). This implies, $J_{n+1}^*(x,h)=J_{n+1}(g^{\infty},x,h)=(n+1)\rho(x)+h(x)$.
\end{proof}

\begin{remark}
We note that in Definition \ref{dfn.canonical}, the condition $(\rho,h)$ are bounded continuous and non-negative can be relaxed to continuous and non-negative. In this case, if \eqref{canonical.triplet.eq} holds, i.e., $(\rho,h,\varphi)$ is a canonical triplet then \eqref{can.tripl.a}-\eqref{can.tripl.c} hold.
\end{remark}

Due to the fact that the average cost as an optimality criterion is underselective, i.e., with limitations in distinguishing optimal policies with different costs, we introduce next a more selective criterion. For other underselective and overselective optimality criteria see \cite{Flynn1976,Flynn1980}.

\begin{definition}\label{opt.policies}
A policy $g^\dagger$ is said to be
\begin{itemize}
\item[(a)] \cite{dynkin79} Strong average cost optimal if \begin{equation}
J(g^\dagger,x)\leq \liminf_{n\rightarrow \infty}\frac{1}{n}J_n(g,x),\quad \forall g\in G,\ x\in{\cal X}.
\end{equation}
\item[(b)] \cite{Flynn1980} F-strong average cost optimal if \begin{equation}
\lim_{n\rightarrow\infty}\frac{1}{n}\Big(J_n(g^\dagger,x)-J_n^*(x)\Big)=0,\quad \forall x\in{\cal X}
\end{equation} where $J_n^*(x)=\inf_{g\in G}J_n(g,x)$.\\
\end{itemize}
\end{definition}
Based on Definition \ref{opt.policies}, next we derive stronger results.

\begin{theorem} \cite{hern1996discrete}\label{main.th,triplet}
Suppose the cost function $f$ satisfies Assumption \ref{assump.cost}, and let $(\rho,h,\varphi)$ be a canonical triplet (with $h$ not necessarily bounded).
\begin{enumerate}
\item[(a)] If for every $g\in G$ and $x\in\cal X$
\begin{equation}\label{main.th.part.aa1} 
\lim_{n\rightarrow \infty}\mathbb{E}_x^{g,Q^*}\Big\{\frac{h(x_n)}{n}\Big\}=0
\end{equation}
then $g^{\infty}$ is an average cost optimal policy and $\rho$ is the average cost value function
\begin{equation}\label{main.th.part.a1} 
J^*(x)=\rho(x)=J(g^{\infty},x)=\lim_{n\rightarrow\infty}\frac{1}{n}J_n(g^{\infty},x),\quad \forall x.
\end{equation}
\item[(b)] If for every $x\in\cal X$ 
\begin{equation}\label{main.th.part.b1}
\lim_{n\rightarrow \infty}\sup_{g\in G}\mathbb{E}_x^{g,Q^*}\Big\{\frac{h(x_n)}{n}\Big\}=0
\end{equation}
then $g^{\infty}$ is strong average cost optimal and F-strong average cost optimal and
\begin{equation}
J^*(x)=\lim_{n\rightarrow\infty}\frac{1}{n}J_n^*(x).
\end{equation}
\een
\end{theorem}
\begin{proof} 
(a)  From \eqref{gen.finite.cost}-\eqref{gen.infinite.cost} and the last equality in \eqref{canonical.triplet.eq}
\begin{eqnarray*}
n\rho(x)+h(x)&=&J_n^*(x,h)
\leq  J_n(g,x)
+\mathbb{E}_x^{g,Q^*}\Big\{h(x_n)\Big\}, \quad \forall g\in G, x\in\cal X.
\end{eqnarray*}
Hence, multiplying by $1/n$, taking the $\limsup$ as $n\rightarrow\infty$, by virtue of \eqref{main.th.part.aa1}, we have $\rho(x)\leq   J(g,x)$, $\forall g,x $
which implies 
\begin{equation}\label{main.thm.proof.triplet.eq1}
\rho(x)\leq  J^*(x),\quad \forall x.
\end{equation}
Furthermore, from \eqref{canonical.triplet.eq} again
\begin{equation}\label{main.thm.proof.triplet.eq11}
J_n(g^{\infty},x,h)=J_n(g^{\infty},x)+\mathbb{E}_x^{g^{\infty},Q^*}\{h(x_n)\}=n\rho(x)+h(x).
\end{equation}
Finally, multiplying both sides of \eqref{main.thm.proof.triplet.eq11} by $1/n$ and then taking both $\limsup$ and $\liminf$ as $n\rightarrow \infty$, we obtain the last two equalities in \eqref{main.th.part.a1}, which in turn, together with \eqref{main.thm.proof.triplet.eq1}, yield the first one since $J^*(x)\leq J(g^{\infty},x)$.

(b) The first equality in \eqref{canonical.triplet.eq} gives
\begin{equation} \label{main.thm.proof.triplet.eq2}
J_n^*(x,h)=J_n(g^{\infty},x)+\mathbb{E}_x^{g^{\infty},Q^*}\{h(x_n)\}.
\end{equation}
On the other hand, by \eqref{gen.finite.cost}-\eqref{gen.infinite.cost}
\begin{eqnarray*}
J_n^*(x,h)=\inf_{g\in G}\Big(J_n(g,x)+\mathbb{E}_x^{g,Q^*}\{h(x_n)\}\Big)\leq & J_n^*(x)+\sup_{g\in G}\mathbb{E}_x^{g,Q^*}\{h(x_n)\}.
\end{eqnarray*}
Thus,
\begin{equation}
0\leq J_n(g^{\infty},x)-J_n^*(x)\leq \sup_{g\in G}\mathbb{E}_x^{g,Q^*}\{h(x_n)\}-\mathbb{E}_x^{g^{\infty},Q^*}\{h(x_n)\}.
\end{equation}
Hence, if $h$ satisfies \eqref{main.th.part.b1}, then $g^{\infty}$ is F-strong average cost optimal.
Finally, to prove that $g^{\infty}$ is strong average cost optimal, we use \eqref{main.thm.proof.triplet.eq2} again to obtain 
\begin{equation*}
J_n(g^{\infty},x)+\mathbb{E}_x^{g^{\infty},Q^*}\{h(x_n)\}\leq J_n(g,x)+\mathbb{E}_x^{g,Q^*}\{h(x_n)\},\quad\forall g,x,n
\end{equation*}
so that from \eqref{main.th.part.b1}
\begin{equation}
\liminf_{n\rightarrow \infty}\frac{1}{n} J_n(g^{\infty},x)\leq \liminf_{n\rightarrow\infty} \frac{1}{n}J_n(g,x).
\end{equation}
Since the left-hand side equals to $J(g^{\infty},x)$ (see \eqref{main.th.part.a1}) it follows that $g^{\infty}$ is indeed strong average cost optimal and the proof is complete.
\end{proof}

Note that, in the case in which $\rho(\cdot)$ is constant, that is $\rho$ does not vary with $x$, then the first optimality equation of Theorem \ref{theorme} is redundant and hence \eqref{can.tripl.a}-\eqref{can.tripl.c} reduce to
\begin{eqnarray}
\quad \rho^*+h(x)&=&\inf_{u\in {\cal U}(x)}\Big\{f(x,u){+}\int_{ \mathclap{\ {\cal X}}}\ h(z)Q^o(dz|x,u){+}\frac{R}{2}\Big(\sup_{z\in{\cal X}}h(z){-}\inf_{z\in{\cal X}}h(z)\Big)\Big\}\label{extr.case1}\\
\rho^* +h(x)&=&f(x,\varphi){+}\int_{ \mathclap{\ {\cal X}}}\ h(z)Q^o(dz|x,\varphi){+}\frac{R}{2}\Big(\sup_{z\in{\cal X}}h(z){-}\inf_{z\in{\cal X}}h(z)\Big).\label{extr.case11}
\end{eqnarray}

Next, we use equations \eqref{can.tripl.equiv.a}-\eqref{can.tripl.equiv.c} of Theorem \ref{theorme} to develop a general policy iteration algorithm for average cost dynamic programming.
\subsection{General Policy Iteration Algorithm for Finite Alphabet Spaces}\label{subsec.inf_ave_gener_policy_iteration_alg}
In this section, we provide a policy iteration algorithm to obtain average cost optimal policies, in which 
 policy evaluation and policy improvement steps are evaluated using the maximizing conditional distribution given by \eqref{Q.matrix.limited.a}. The proposed algorithm is considerably more complex compared to Algorithm \ref{alg.pol.iter.aver.cost}. Nevertheless, it solves the MDP for all range of values of total variation parameter $R\in[0,2]$, and without imposing the irreducibility condition, as in Section \ref{subsec.inf_ave_policy_iteration_alg}.\\

\begin{algorithm}
(General policy iteration)
\ben
\item[1)] Let $m=0$ and select an arbitrary stationary Markov control policy $g_0:\cal X \longmapsto {\cal U}$.
\item[2)] (Policy Evaluation) Solve the equations
\begin{align}
J_{Q^o}(g_m)&=Q^o(g_m)J_{Q^o}(g_m)\label{general.pol.eval.J}\\
J_{Q^o}(g_m)+h_{Q^o}(g_m)&=f(g_m)+Q^o(g_m)h_{Q^o}(g_m)\label{general.pol.eval.V}
\end{align}
for $J_{Q^o}(g_m)$ and $h_{Q^o}(g_m)$. Identify the support sets of \eqref{general.pol.eval.V} using \eqref{Q.matrix.limited.b} (where $h$ replaces $V$), and construct the matrix $Q^{*}(g_m)$ using \eqref{Q.matrix.limited.a}. Solve the equations
\begin{align}
J_{Q^{*}}(g_m)&{=}Q^{*}(g_m)J_{Q^{*}}(g_m)\label{general.pol.eval.Q*.J}\\
J_{Q^{*}}(g_m){+}h_{Q^{*}}(g_m)&{=}f(g_m){+}Q^{*}(g_m)h_{Q^{*}}(g_m)\label{general.pol.eval.Q*.V}
\end{align}
 for $J_{Q^{*}}(g_m)$ and $h_{Q^{*}}(g_m)$.

\item[3)] (Policy Improvement) \\
a) Let \begin{equation}
g_{m+1}=\argmin_{g\in\mathbb{R}^{|{\cal X}|}} \big\{Q^{*}(g)J_{Q^{*}}(g_m)\big\}.
\end{equation} 
If $g_{m+1}=g_m$ go to step 3b); otherwise let $m=m+1$ and return to step 2.

b)  Let 
\begin{equation}
g_{m+1}=\argmin_{g\in\mathbb{R}^{|{\cal X}|}} \big\{f(g)+Q^{*}(g)h_{Q^{*}}(g_m)\big\}.
\end{equation} 
\item[4)] If $g_{m+1}=g_m$, let $g^*=g_m$; else let $m=m+1$ and return to step 2.  \\
\een
\label{alg.general.pol.iter.aver.cost}
\end{algorithm}

For MCP with finite state and action spaces the proposed general policy iteration algorithm converges in a finite number of iterations. However, for MCP on Borel spaces the proposed policy iteration algorithm might not converge, or it might converge to a suboptimal  value, and hence one must introduce additional assumptions (i.e., see \cite{Lerma-Lasserre1997,Meyn1997}).
In Section \ref{working example4}, we illustrate through an example how Algorithm \ref{alg.general.pol.iter.aver.cost} is applied.

\section{Examples}\label{sec.examples}
In this section we illustrate the new dynamic programming equations and the corresponding policy iteration algorithms through examples. In particular, in Section \ref{working example3} we present an application of the infinite horizon minimax problem for average cost by employing policy iteration Algorithm \ref{alg.pol.iter.aver.cost}, and in Section \ref{working example4} we present an application of the infinite horizon minimax problem for average cost by employing policy iteration Algorithm \ref{alg.general.pol.iter.aver.cost}. The essential difference between the two examples is that the MDP of the latter is described by a transition probability graph which is reducible.

\subsection{Infinite Horizon Minimax MDP - Policy Iteration Algorithm \ref{alg.pol.iter.aver.cost}}\label{working example3}
Here, we illustrate an application of the infinite horizon minimax problem for average cost, by considering the stochastic control system as shown in Fig.\ref{fig.exd.11a}, with state space ${\cal X}=\{1,2,3\}$ and control set ${\cal U}=\{u_1,u_2\}$. Assume that the nominal transition probabilities under controls $u_1$ and $u_2$ are given  by
\begin{eqnarray}\label{trans.prob.inf.hor.dmcm}
  Q^o(u_1){=}\frac{1}{9}\left(
                    \begin{array}{ccc}
                      3 & 1 & 5 \\
                      4 & 2 & 3 \\
                      1 & 6 & 2 \\
                    \end{array}
                  \right),\quad Q^o(u_2){=}\frac{1}{9}\left(
                    \begin{array}{ccc}
                      1 & 2 & 6 \\
                      4 & 2 & 3 \\
                      4 & 1 & 4 \\
                    \end{array}
                  \right)
\end{eqnarray}
the total variation distance radius is $R=6/9$, and the cost function under each state and action is \begin{equation*}
f(1,u_1)=2, \ f(2,u_1)=1,\  f(3,u_1)=3,\
 f(1,u_2)=0.5, \ f(2,u_2)=3, \ f(3,u_2)=0.
\end{equation*}
To obtain an optimal stationary policy of the infinite horizon minimax problem for average cost, policy iteration algorithm \ref{alg.pol.iter.aver.cost} is applied.

\noindent \textbf{A. Let $m=0$.} 

1) Select the initial policies as follows $g_0(1)=u_1$, $g_0(2)=u_2$, $g_0(3)=u_2$.

2) Solve the equation $J_{Q^o}(g_0)e+V_{Q^o}(g_0)=f(g_0)+Q^o(g_0)V_{Q^o}(g_0)$ for $J_{Q^o}(g_0)\in\mathbb{R}$ and $V_{Q^o}(g_0)\in\mathbb{R}^3$, which is given by
\begin{equation*}
  J_{Q^o}(g_0)\left(\begin{array}{c}
     1 \\
     1 \\
     1
   \end{array}\right)+\left(\begin{array}{c}
                                     V_{Q^o}(g_0,1) \\
                                     V_{Q^o}(g_0,2) \\
                                     V_{Q^o}(g_0,3)
                                   \end{array}\right) =\left(\begin{array}{c}
     2 \\
     3 \\
     0
   \end{array}\right)+\frac{1}{9}\left(
                              \begin{array}{ccc}
                                3 & 1 & 5 \\
                                4 & 2 & 3 \\
                                4 & 1 & 4 \\
                              \end{array}
                            \right)\left(\begin{array}{c}
                                     V_{Q^o}(g_0,1) \\
                                     V_{Q^o}(g_0,2) \\
                                     V_{Q^o}(g_0,3)
                                   \end{array}\right).
\end{equation*}
Since $V_{Q^o}(g_0)$ is uniquely determined up to an additive constant, let $V_{Q^o}(g_0,3)=0$. The solution is
\begin{equation*}
\left(\begin{array}{c}
                                     V_{Q^o}(g_0,1) \\
                                     V_{Q^o}(g_0,2) \\
                                     V_{Q^o}(g_0,3)
                                   \end{array}\right)=\left(\begin{array}{c}
                                     1.8 \\
                                     3.375 \\
                                     0
                                   \end{array}\right),\quad J_{Q^o}(g_0)=1.175.
\end{equation*}
Note that, $V_{Q^o}\triangleq\{V_{Q^o}(1),V_{Q^o}(2),V_{Q^o}(3)\}$, $|{\cal X}|=3$, and hence 
\begin{eqnarray*}
  {\cal X}^+\triangleq\{x\in{\cal X}:V_{Q^o}(x)&=&\max\{V_{Q^o}(x):x\in{\cal X}\}\}\\
  &=&\{x\in{\cal X}:V_{Q^o}(x)=V_{Q^o}(2)\}=\{2\}\\
  {\cal X}^-\triangleq  \{x\in{\cal X}:V_{Q^o}(x)&=&\min\{V_{Q^o}(x):x\in{\cal X}\}\}\\
  &=&\{x\in{\cal X}:V_{Q^o}(x)=V_{Q^o}(3)\}=\{3\}\\
  {\cal X}_1\triangleq \{x\in{\cal X}:V_{Q^o}(x)&=&\min\{V_{Q^o}(\alpha):\alpha\in {\cal X}{\setminus} {\cal X}^+{\cup}{\cal X}^-\}\}\\
  &=&\{x\in{\cal X}:V_{Q^o}(x)=V_{Q^o}(1)\}=\{1\}.
\end{eqnarray*}
Once the partition is been identified, (\ref{Q.matrix.limited.a}) is applied to obtain (\ref{initQ1}) and (\ref{initQ2}).
\begin{eqnarray}
Q^*(u_1)&=&\begin{pmatrix}
                     \left(q^o_{11}(u_1)-\left(\frac{R}{2}-q^o_{13}(u_1)\right)^+\right)^+ & \min\left(1,q^o_{12}(u_1)+\frac{R}{2}\right) & \left(q^o_{13}(u_1)-\frac{R}{2}\right)^+ \\
                     \left(q^o_{21}(u_1)-\left(\frac{R}{2}-q^o_{23}(u_1)\right)^+\right)^+ & \min\left(1,q^o_{22}(u_1)+\frac{R}{2}\right) & \left(q^o_{23}(u_1)-\frac{R}{2}\right)^+ \\
                     \left(q^o_{31}(u_1)-\left(\frac{R}{2}-q^o_{33}(u_1)\right)^+\right)^+ & \min\left(1,q^o_{32}(u_1)+\frac{R}{2}\right) & \left(q^o_{33}(u_1)-\frac{R}{2}\right)^+ 
                   \end{pmatrix} \nonumber\\ [-1.5ex]\label{initQ1}\\[-1.5ex]&=&\frac{1}{9}
                              \begin{pmatrix}
                                3 & 4 & 2 \\
                                4 & 5 & 0 \\
                                0 & 9 & 0 
                              \end{pmatrix}.\nonumber \\ \nonumber\\
Q^*(u_2)&=&\begin{pmatrix}
\left(q^o_{11}(u_2)-\left(\frac{R}{2}-q^o_{13}(u_2)\right)^+\right)^+ & \min\left(1,q^o_{12}(u_2)+\frac{R}{2}\right) & \left(q^o_{13}(u_2)-\frac{R}{2}\right)^+ \\
                     \left(q^o_{21}(u_2)-\left(\frac{R}{2}-q^o_{23}(u_2)\right)^+\right)^+ & \min\left(1,q^o_{22}(u_2)+\frac{R}{2}\right) & \left(q^o_{23}(u_2)-\frac{R}{2}\right)^+ \\
                     \left(q^o_{31}(u_2)-\left(\frac{R}{2}-q^o_{33}(u_2)\right)^+\right)^+ & \min\left(1,q^o_{32}(u_2)+\frac{R}{2}\right) & \left(q^o_{33}(u_2)-\frac{R}{2}\right)^+ 
\end{pmatrix}\nonumber\\ [-1.5ex]\label{initQ2}\\[-1.5ex]&=&\frac{1}{9}\begin{pmatrix}
1 & 5 & 3 \\
4 & 5 & 0 \\
4 & 4 & 1 
\end{pmatrix}.\nonumber
\end{eqnarray}
The transition probability graph of $Q^*$, under controls $u_1$ and $u_2$, is depicted in Fig.\ref{ch.dyn.work.ex.3.fig2}. Note that, since every state can reach every other state, matrix $Q^*(u)$ remains irreducible under both controls.
\begin{figure}[!htbp]
\centering
\subfloat[Matrix $Q^*$ under control $u_1$.]{
\label{ch.dyn.work.ex.3.fig1.1} 
\includegraphics[height=4.9cm, width=.49\linewidth]{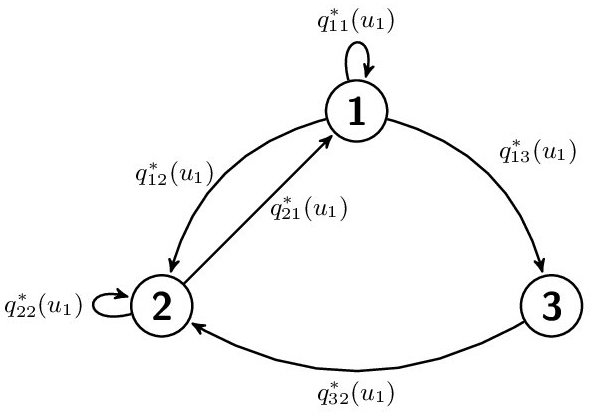}}
\subfloat[Matrix $Q^*$ under control $u_2$.]{
\label{ch.dyn.work.ex.3.fig1.2} 
\includegraphics[height=4.9cm,width=.49\linewidth]{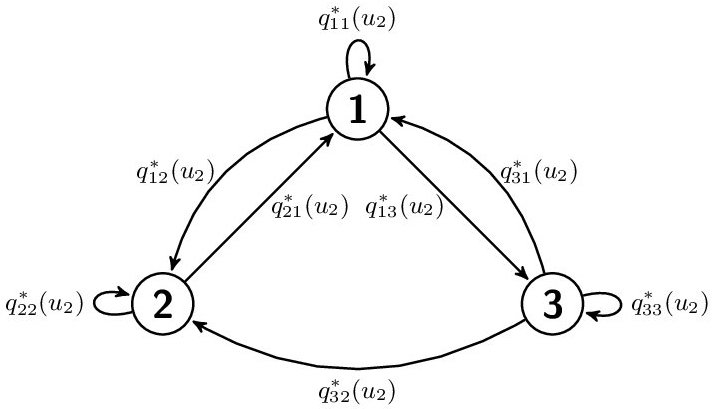}}\vspace{-.2cm}
\caption[Transition Probability Graph of $Q^*$ under controls $u_1$ and $u_2$.]{Transition Probability Graph of $Q^*$ under controls $u_1$ and $u_2$.}
\label{ch.dyn.work.ex.3.fig2}
\end{figure}
Next, we proceed to solve the equation $J_{Q^*}(g_0)e+V_{Q^*}(g_0)=f(g_0)+Q^*(g_0)V_{Q^*}(g_0)$ for $J_{Q^*}(g_0)\in\mathbb{R}$ and $V_{Q^*}(g_0)\in\mathbb{R}^3$, which is given by
\begin{equation*}
  J_{Q^*}( g_0)\left(\begin{array}{c}
     1 \\
     1 \\
     1
   \end{array}\right)+\left(\begin{array}{c}
                                     V_{Q^*}(g_0,1) \\
                                     V_{Q^*}(g_0,2) \\
                                     V_{Q^*}(g_0,3)
                                   \end{array}\right) =\left(\begin{array}{c}
     2 \\
     3 \\
     0
   \end{array}\right)+\frac{1}{9}\left(
                              \begin{array}{ccc}
                                3 & 4 & 2 \\
                                4 & 5 & 0 \\
                                4 & 4 & 1 \\
                              \end{array}
                            \right)\left(\begin{array}{c}
                                     V_{Q^*}(g_0,1) \\
                                     V_{Q^*}(g_0,2) \\
                                     V_{Q^*}(g_0,3)
                                   \end{array}\right).
\end{equation*} 
Since $V_{Q^*}(g_0)$ is uniquely determined up to an additive constant, let $V_{Q^*}(g_0,3)=0$. The solution is
\begin{equation*}
\left(\begin{array}{c}
                                     V_{Q^*}(g_0,1) \\
                                     V_{Q^*}(g_0,2) \\
                                     V_{Q^*}(g_0,3)
                                   \end{array}\right)=\left(\begin{array}{c}
                                     1.8 \\
                                     3.375 \\
                                     0
                                   \end{array}\right),\quad J_{Q^*}(g_0)=2.3.
\end{equation*}

3) Let $g_1=\argmin_{g\in\mathbb{R}^3}\{f(g)+Q^*(g)V_{Q^*}(g_0)\}$. Then
\begin{eqnarray*}
g_1(1)&=&\argmin\Big\{f(1,u_1)+
q^*_{11}(u_1)V_{Q^*}(g_0,1){+}q^*_{12}(u_1)V_{Q^*}(g_0,2){+}q^*_{13}(u_1)V_{Q^*}(g_0,3),
\\
&&\qquad \qquad f(1,u_2)+q^*_{11}(u_2)V_{Q^*}(g_0,1){+}q^*_{12}(u_2)V_{Q^*}(g_0,2){+}q^*_{13}(u_2)V_{Q^*}(g_0,3)\Big\}\\
&=&\argmin\Big\{4.099,2.573\Big\}=\{2\}\Longrightarrow  g_1(1)=u_2.
\end{eqnarray*}
Following a similar procedure for the rest we obtain the following.
\begin{eqnarray*}
g_1(2)&=&\argmin\Big\{3.673,5.673\Big\}=\{1\}\Longrightarrow  g_1(2)=u_1\\
g_1(3)&=&\argmin\Big\{6.375,2.3\Big\}=\{2\}\Longrightarrow  g_1(3)=u_2.
\end{eqnarray*}
Since, $g_1\neq g_0$, let $m=1$ and return to step 2.\\

\noindent \textbf{B. Let $m=1$.}

2) Solve the equation $J_{Q^o}(g_1)e+V_{Q^o}(g_1)=f(g_1)+Q^o(g_1)V_{Q^o}(g_1)$, $V_{Q^o}(g_1,3)=0$, for $J_{Q^o}(g_1)\in\mathbb{R}$ and $V_{Q^o}(g_1)\in\mathbb{R}^3$. The solution is 
\begin{equation*}
\left(\begin{array}{c}
                                     V_{Q^o}(g_1,1) \\
                                     V_{Q^o}(g_1,2) \\
                                     V_{Q^o}(g_1,3)
                                   \end{array}\right)=\left(\begin{array}{c}
                                     0.468 \\
                                     1.125 \\
                                     0
                                   \end{array}\right),\quad J_{Q^o}(g_1)=0.333.
\end{equation*}
Therefore, ${\cal X}^+=\{2\}$, ${\cal X}^-=\{3\}$ and ${\cal X}_1=\{1\}$. Since the partition is the same as in $m=0$ then $Q^*(u_1)$ and $Q^*(u_2)$ are given by \eqref{initQ1} and \eqref{initQ2}, respectively.

Solve the equation $J_{Q^*}(g_1)e+V_{Q^*}(g_1)=f(g_1)+Q^*(g_1)V_{Q^*}(g_1)$, $V_{Q^*}(g_1,3)=0$, for $J_{Q^*}(g_1)\in\mathbb{R}$ and $V_{Q^*}(g_1)\in\mathbb{R}^3$. The solution is
\begin{equation*}
\left(\begin{array}{c}
                                     V_{Q^*}(g_1,1) \\
                                     V_{Q^*}(g_1,2) \\
                                     V_{Q^*}(g_1,3)
                                   \end{array}\right)=\left(\begin{array}{c}
                                     0.468 \\
                                     1.125 \\
                                     0
                                   \end{array}\right),\quad J_{Q^*}(g_1)=0.708.
\end{equation*}

3) Let $g_2=\argmin_{g\in\mathbb{R}^3}\{f(g)+Q^*(g)V_{Q^*}(g_1)\}$. Then
\begin{eqnarray*}
g_2(1)&=&\argmin\Big\{2.656,1.177\Big\}=\{2\}\Longrightarrow  g_2(1)=u_2\\
g_2(2)&=&\argmin\Big\{1.831,3.831\Big\}=\{1\}\Longrightarrow  g_2(2)=u_1\\
g_2(3)&=&\argmin\Big\{4.125,0.708\Big\}=\{2\}\Longrightarrow  g_2(3)=u_2.
\end{eqnarray*}

4) Since, $g_2=g_1$, then $g^*=g_1$ is an optimal control policy with $J_{Q^*}=0.708$, $V_{Q^*}(1)=0.468$, $V_{Q^*}(2)=1.125$ and $V_{Q^*}(3)=0$.

\subsection{Infinite Horizon Minimax MDP - General Policy Iteration Algorithm \ref{alg.general.pol.iter.aver.cost}}\label{working example4}
In this example, we illustrate an application of the infinite horizon minimax problem for average cost, by considering the stochastic control system shown in Fig.\ref{fig.exd.11a}, with ${\cal X}=\{1,2,3\}$ and control set ${\cal U}=\{u_1,u_2\}$. The essential difference between this example and the previous one, is that here, the stochastic control system under consideration is described by a transition probability graph which is reducible, and hence general policy iteration algorithm \ref{alg.general.pol.iter.aver.cost} is applied.

Assume that the nominal transition probabilities under controls $u_1$ and $u_2$ are given  by
\begin{eqnarray}\label{trans.prob.b.inf.hor.dmcm}
  Q^o(u_1){=}\frac{1}{9}\left(
                    \begin{array}{ccc}
                      0 & 5 & 4 \\
                      0 & 9 & 0 \\
                      0 & 0 & 9 \\
                    \end{array}
                  \right),\quad Q^o(u_2){=}\frac{1}{9}\left(
                    \begin{array}{ccc}
                      2 & 7 & 0 \\
                      3 & 6 & 0 \\
                      8 & 0 & 1 \\
                    \end{array}
                  \right)
\end{eqnarray}
the total variation distance radius is $R=14/9$, and the cost function under each state and action is
\begin{equation*}
  f(1,u_1)=2,\  f(2,u_1)=1,\ f(3,u_1)=3,\ f(1,u_2)=0.5,\  f(2,u_2)=3, \ f(3,u_2)=0.
\end{equation*}

 \noindent \textbf{A. Let $m=0$.} 

1) Select the initial policies as follows $g_0(1)=u_1$, $g_0(2)=u_1$, $g_0(3)=u_1$.

2) Solve the equation $J_{Q^o}(g_0)=Q^o(g_0)J_{Q^o}(g_0)$. The optimality equations \eqref{general.pol.eval.J} are 
\begin{subequations}\label{opt.eq.J.first}
\begin{eqnarray}
J_{Q^o}(g_0,1)&=&\frac{5}{9}J_{Q^o}(g_0,2)+\frac{4}{9}J_{Q^o}(g_0,3)\\
J_{Q^o}(g_0,2)&=&J_{Q^o}(g_0,2)\\
J_{Q^o}(g_0,3)&=&J_{Q^o}(g_0,3).
\end{eqnarray}
\end{subequations}
Next, solve the equation $J_{Q^o}(g_0)+h_{Q^o}(g_0)=f(g_0)+Q^o(g_0)h_{Q^o}(g_0)$, for $J_{Q^o}(g_0)\in\mathbb{R}^3$ and $h_{Q^o}(g_0)\in\mathbb{R}^3$. The optimality equations \eqref{general.pol.eval.V} are given by
\begin{subequations}\label{opt.eq.V.first}
\begin{eqnarray}
J_{Q^o}(g_0,1)+h_{Q^o}(g_0,1)&=&2+\frac{5}{9}h_{Q^o}(g_0,2)+\frac{4}{9}h_{Q^o}(g_0,3)\\
J_{Q^o}(g_0,2)+h_{Q^o}(g_0,2)&=&1+h_{Q^o}(g_0,2)\\
J_{Q^o}(g_0,3)+h_{Q^o}(g_0,3)&=&3+h_{Q^o}(g_0,3).
\end{eqnarray}
\end{subequations}
The solution of \eqref{opt.eq.J.first} and \eqref{opt.eq.V.first} has 
\begin{align*}
h_{Q^o}(g_0,1)&{=}\frac{1}{9}{+}\frac{5}{9}\alpha{+}\frac{4}{9}\beta,   &   h_{Q^o}(g_0,2)&{=}\alpha, &  h_{Q^o}(g_0,3)&{=}\beta,\\
J_{Q^o}(g_0,1)&{=}1.888,   &    J_{Q^o}(g_0,2)&{=}1, & J_{Q^o}(g_0,3)&{=}3.
\end{align*}
Setting $\alpha=1$ and $\beta=0$ (arbitrary constants) yields
\begin{equation*}
h_{Q^o}(g_0,1)=0.666,\quad h_{Q^o}(g_0,2)=1,\quad h_{Q^o}(g_0,3)=0.
\end{equation*}
Note that, $h_{Q^o}=\{h_{Q^o}(1),h_{Q^o}(2),h_{Q^o}(3)\}$, and hence the support sets based on the values of $h_{Q^o}$ are ${\cal X}^+=\{2\}$, ${\cal X}^-=\{3\}$ and ${\cal X}_1=\{1\}$. Once the partition is been identified, \eqref{Q.matrix.limited.a} is applied to obtain \eqref{Q.matrix.Vu1} and \eqref{Q.matrix.Vu2}.
\begin{eqnarray}
Q^{*}(u_1)&=&\begin{pmatrix}
                     \left(q^o_{11}(u_1)-\left(\frac{R}{2}-q^o_{13}(u_1)\right)^+\right)^+ & \min\left(1,q^o_{12}(u_1)+\frac{R}{2}\right) & \left(q^o_{13}(u_1)-\frac{R}{2}\right)^+ \\
                     \left(q^o_{21}(u_1)-\left(\frac{R}{2}-q^o_{23}(u_1)\right)^+\right)^+ & \min\left(1,q^o_{22}(u_1)+\frac{R}{2}\right) & \left(q^o_{23}(u_1)-\frac{R}{2}\right)^+ \\
                     \left(q^o_{31}(u_1)-\left(\frac{R}{2}-q^o_{33}(u_1)\right)^+\right)^+ & \min\left(1,q^o_{32}(u_1)+\frac{R}{2}\right) & \left(q^o_{33}(u_1)-\frac{R}{2}\right)^+ 
                   \end{pmatrix}\nonumber\\ [-1.5ex]\label{Q.matrix.Vu1}\\[-1.5ex] &=&\frac{1}{9}
                              \begin{pmatrix}
                                0 & 9 & 0 \\
                                0 & 9 & 0 \\
                                0 & 7 & 2 
                              \end{pmatrix}\nonumber \\ \nonumber\\
Q^{*}(u_2)&=&\begin{pmatrix}
\left(q^o_{11}(u_2)-\left(\frac{R}{2}-q^o_{13}(u_2)\right)^+\right)^+ & \min\left(1,q^o_{12}(u_2)+\frac{R}{2}\right) & \left(q^o_{13}(u_2)-\frac{R}{2}\right)^+ \\
                     \left(q^o_{21}(u_2)-\left(\frac{R}{2}-q^o_{23}(u_2)\right)^+\right)^+ & \min\left(1,q^o_{22}(u_2)+\frac{R}{2}\right) & \left(q^o_{23}(u_2)-\frac{R}{2}\right)^+ \\
                     \left(q^o_{31}(u_2)-\left(\frac{R}{2}-q^o_{33}(u_2)\right)^+\right)^+ & \min\left(1,q^o_{32}(u_2)+\frac{R}{2}\right) & \left(q^o_{33}(u_2)-\frac{R}{2}\right)^+ 
\end{pmatrix}\nonumber\\ [-1.5ex]\label{Q.matrix.Vu2}\\[-1.5ex]&=&\frac{1}{9}\begin{pmatrix}
0 & 9 & 0 \\
0 & 9 & 0 \\
2 & 7 & 0 
\end{pmatrix}.\nonumber
\end{eqnarray}

Next, solve the equation $J_{Q^{*}}(g_0)=Q^{*}(g_0)J_{Q^{*}}(g_0)$. The optimality equations \eqref{general.pol.eval.Q*.J} are 
\begin{subequations}\label{opt.eq.J.opt.first}
\begin{eqnarray}
J_{Q^{*}}(g_0,1)&=&J_{Q^{*}}(g_0,2)\\
J_{Q^{*}}(g_0,2)&=&J_{Q^{*}}(g_0,2)\\
J_{Q^{*}}(g_0,3)&=&\frac{7}{9}J_{Q^{*}}(g_0,2)+\frac{2}{9}J_{Q^{*}}(g_0,3)
\end{eqnarray}
\end{subequations}
and hence, $J_{Q^{*}}(g_0,1)=J_{Q^{*}}(g_0,2)=J_{Q^{*}}(g_0,3)$.

Next, solve the equation $J_{Q^{*}}(g_0)+h_{Q^{*}}(g_0)=f(g_0)+Q^{*}(g_0)h_{Q^{*}}(g_0)$, for $J_{Q^{*}}(g_0)\in\mathbb{R}^3$ and $h_{Q^{*}}(g_0)\in\mathbb{R}^3$. The optimality equations \eqref{general.pol.eval.Q*.V} are given by
\begin{subequations}\label{opt.eq.V.opt.first}
\begin{eqnarray}
J_{Q^{*}}(g_0,1)+h_{Q^{*}}(g_0,1)&=&2+h_{Q^{*}}(g_0,2)\\
J_{Q^{*}}(g_0,2)+h_{Q^{*}}(g_0,2)&=&1+h_{Q^{*}}(g_0,2)\\
J_{Q^{*}}(g_0,3)+\frac{7}{9} h_{Q^{*}}(g_0,3)&=&3+\frac{7}{9}h_{Q^{*}}(g_0,2).
\end{eqnarray}
\end{subequations}
The solution of \eqref{opt.eq.J.opt.first} and \eqref{opt.eq.V.opt.first} has
\begin{align*}
h_{Q^{*}}(g_0,1)&=1{+}\alpha, & h_{Q^{*}}(g_0,2)&=\alpha, & h_{Q^{*}}(g_0,3)&=\frac{18}{7}{+}\alpha,\\
J_{Q^{*}}(g_0,1)&=1, & J_{Q^{*}}(g_0,2)&=1, & J_{Q^{*}}(g_0,3)&=1.
\end{align*}
Setting $\alpha=1$ (arbitrary constant) yields
\begin{equation*}
h_{Q^{*}}(g_0,1)=2,\quad h_{Q^{*}}(g_0,2)=1,\quad h_{Q^{*}}(g_0,3)=3.57.
\end{equation*}

3) a) Since $J_{Q^{*}}(g_0,1)=J_{Q^{*}}(g_0,2)=J_{Q^{*}}(g_0,3)$, then clearly $g_1=g_0$ and we proceed to step 3b).

b)  Let $g_1=\argmin_{g\in\mathbb{R}^3}\{f(g)+Q^{*}(g)h_{Q^{*}}(g_0)\}$, then the resulting control policies are $g_1(1)=u_2$, $ g_1(2)=u_1$ and $g_1(3)=u_2$. Since $g_1\neq g_0$, let $m=1$ and return to step 2.\\

\noindent \textbf{B. Let $m=1$.}

2) Solve the equation $J_{Q^o}(g_1)=Q^o(g_1)J_{Q^o}(g_1)$. The optimality equations \eqref{general.pol.eval.J} are 
\begin{subequations}\label{opt.eq.J.second}
\begin{eqnarray}
J_{Q^o}(g_1,1)&=&\frac{2}{9}J_{Q^o}(g_1,1)+\frac{7}{9}J_{Q^o}(g_1,2)\\
J_{Q^o}(g_1,2)&=&J_{Q^o}(g_1,2)\\
J_{Q^o}(g_1,3)&=&\frac{8}{9}J_{Q^o}(g_1,1)+\frac{1}{9}J_{Q^o}(g_1,3)
\end{eqnarray}
\end{subequations}
and hence, $J_{Q^o}(g_1,1)=J_{Q^o}(g_1,2)=J_{Q^o}(g_1,3)$.

Next, solve the equation $J_{Q^o}(g_1)+h_{Q^o}(g_1)=f(g_1)+Q^o(g_1)h_{Q^o}(g_1)$, for $J_{Q^o}(g_1)\in\mathbb{R}^3$ and $h_{Q^o}(g_1)\in\mathbb{R}^3$. The optimality equations \eqref{general.pol.eval.V} are given by
\begin{subequations}\label{opt.eq.V.second}
\begin{eqnarray}
J_{Q^o}(g_1,1)+\frac{7}{9} h_{Q^o}(g_1,1)&=&0.5+\frac{7}{9}h_{Q^o}(g_1,2)\\
J_{Q^o}(g_1,2)+h_{Q^o}(g_1,2)&=&1+h_{Q^o}(g_1,2)\\
J_{Q^o}(g_1,3)+\frac{8}{9} h_{Q^o}(g_1,3)&=&\frac{8}{9}h_{Q^o}(g_1,1).
\end{eqnarray}
\end{subequations}
The solution of \eqref{opt.eq.J.second} and \eqref{opt.eq.V.second} has 
\begin{align*}
h_{Q^o}(g_1,1)&{=}\alpha+\frac{9}{8}, & h_{Q^o}(g_1,2)&{=}\alpha+\frac{99}{56}, & h_{Q^o}(g_1,3)&{=}\alpha,\\
J_{Q^o}(g_1,1)&{=}1, &  J_{Q^o}(g_1,2)&{=}1, & J_{Q^o}(g_1,3)&{=}1.
\end{align*}
Setting $\alpha=1$ (arbitrary constant) yields
\begin{equation*}
h_{Q^o}(g_1,1)=2.125,\quad h_{Q^o}(g_1,2)=2.76,\quad h_{Q^o}(g_1,3)=1.
\end{equation*}
Hence, we proceed with the identification of the support sets, which are ${\cal X}^+=\{2\}$, ${\cal X}^-=\{3\}$ and ${\cal X}_1=\{1\}$. Since the partition is the same as in $m=0$ then $Q^{*}(u_1)$ and $Q^{*}(u_2)$ are equal to \eqref{Q.matrix.Vu1} and \eqref{Q.matrix.Vu2}, respectively.

Next, solve the equation $J_{Q^{*}}(g_1)=Q^{*}(g_1)J_{Q^{*}}(g_1)$. The optimality equations \eqref{general.pol.eval.Q*.J} are 
\begin{subequations}\label{opt.eq.J.opt.second}
\begin{eqnarray}
J_{Q^{*}}(g_1,1)&=&J_{Q^{*}}(g_1,2)\\
J_{Q^{*}}(g_1,2)&=&J_{Q^{*}}(g_1,2)\\
J_{Q^{*}}(g_1,3)&=&\frac{2}{9}J_{Q^{*}}(g_1,1)+\frac{7}{9}J_{Q^{*}}(g_1,2)
\end{eqnarray}
\end{subequations}
and hence, $J_{Q^{*}}(g_1,1)=J_{Q^{*}}(g_1,2)=J_{Q^{*}}(g_1,3)$.

Next, solve the equation $J_{Q^{*}}(g_1)+h_{Q^{*}}(g_1)=f(g_1)+Q^{*}(g_1)h_{Q^{*}}(g_1)$, for $J_{Q^{*}}(g_1)\in\mathbb{R}^3$ and $h_{Q^{*}}(g_1)\in\mathbb{R}^3$. The optimality equations \eqref{general.pol.eval.Q*.V} are given by
\begin{subequations}\label{opt.eq.V.opt.second}
\begin{eqnarray}
J_{Q^{*}}(g_1,1)+h_{Q^{*}}(g_1,1)&=&0.5+h_{Q^{*}}(g_1,2)\\
J_{Q^{*}}(g_1,2)+h_{Q^{*}}(g_1,2)&=&1+h_{Q^{*}}(g_1,2)\\
J_{Q^{*}}(g_1,3)+h_{Q^{*}}(g_1,3)&=&\frac{2}{9}h_{Q^{*}}(g_1,1)+\frac{7}{9}h_{Q^{*}}(g_1,2).
\end{eqnarray}
\end{subequations}
The solution of \eqref{opt.eq.J.opt.second} and \eqref{opt.eq.V.opt.second} has 
\begin{align*}
h_{Q^{*}}(g_1,1)&{=}\alpha+\frac{11}{18}, & h_{Q^{*}}(g_1,2)&{=}\alpha+\frac{10}{9}, & h_{Q^{*}}(g_1,3)&{=}\alpha,\\
J_{Q^{*}}(g_1,1)&{=}1, & J_{Q^{*}}(g_1,2)&{=}1, & J_{Q^{*}}(g_1,3)&{=}1.
\end{align*}
Setting $\alpha=1$ yields
\begin{equation*}
h_{Q^{*}}(g_1,1)=1.611,\quad h_{Q^{*}}(g_1,2)=2.111,\quad h_{Q^{*}}(g_1,3)=1.
\end{equation*}

3) a) Since $J_{Q^{*}}(g_1,1)=J_{Q^{*}}(g_1,2)=J_{Q^{*}}(g_1,3)$, then clearly $g_2=g_1$ and we proceed to step 3b).

b) Let $g_2=\argmin_{g\in\mathbb{R}^3}\{f(g)+Q^*(g)h_{Q^*}(g_1)\}$, the resulting control policies are $g_2(1)=u_2$, $g_2(2)=u_1$ and $g_2(3)=u_2$.

4) Because, $g_2=g_1$, then $g^*=g_1$ is an optimal control policy with $J_{Q^{*}}(1)=J_{Q^{*}}(2)=J_{Q^{*}}(3)=1$, $h_{Q^{*}}(1)=1.611$, $h_{Q^{*}}(2)=2.111$ and $h_{Q^{*}}(3)=1$.

%
%
%
%
\section{Conclusions}\label{sec.conclusions}
In this paper, we examined the optimality of minimax MDP via dynamic programming on an infinite horizon, when the ambiguity class is described by the total variation distance between the conditional distribution of the true controlled process and the conditional distribution of a nominal controlled process. As optimality criterion we considered the average pay-off per unit time. Under the assumption that for every stationary Markov control policy the maximizing stochastic matrix is irreducible, we derived a new dynamic programming equation and a new policy iteration algorithm. However, due to the water-filling behavior of the maximizing conditional distribution, it turns out that our proposed method of solution is limited only to a specific range of values of total variation distance. To circumvent this limit, we consider general Borel spaces, and we derive a general dynamic programming equation by introducing a pair of dynamic programming equations, and, consequently a new policy iteration algorithm, which solve the minimax MDP for all $R\in[0,2]$. Finally, the application of our recommended policy iteration algorithms is shown via illustrative examples.

\bibliographystyle{siam}
\bibliography{ref}

\begin{thebibliography}{10}

\bibitem{Arapostathis93}
{\sc A.~Arapostathis, V.~.S. Borkar, E.~Fernandez-Gaucherand, M.~K. Ghosh, and
  S.~I. Marcus}, {\em Discrete-time controlled {M}arkov processes with average
  cost criterion: a survey}, SIAM J. Control Optim., 31 (1993), pp.~282--344.

\bibitem{Borkar84}
{\sc V.~S. {Borkar}}, {\em On minimum cost per unit time control of {M}arkov
  chains}, SIAM J. Control Optim., 22 (1984), pp.~965--978.

\bibitem{Borkar88}
\leavevmode\vrule height 2pt depth -1.6pt width 23pt, {\em Control of {M}arkov
  chains with long-run average cost criterion}, Stochastic Differential
  systems, Stochastic Control Theory and Applications,  (1988), pp.~57--77.

\bibitem{Borkar89}
\leavevmode\vrule height 2pt depth -1.6pt width 23pt, {\em Control of {M}arkov
  chains with long-run average cost criterion: the dynamic programming
  equations}, SIAM J. Control Optim., 27 (1989), pp.~642--657.

\bibitem{caines88}
{\sc P.~E. {Caines}}, {\em Linear stochastic systems}, John Wiley \& Sons,
  Inc., New York, 1988.

\bibitem{ctc2012}
{\sc C.D. Charalambous, I.~Tzortzis, and T.~Charalambous}, {\em Dynamic
  programming with total variational distance uncertainty}, in 51st IEEE
  Conference on Decision and Control, Maui, Hawaii, Dec. 10--13, 2012.

\bibitem{ctlthem2013j}
{\sc Charalambos~D. Charalambous, I.~Tzortzis, S.~Loyka, and T.~Charalambous},
  {\em {Extremum problems with total variation distance and their
  applications}}, IEEE Trans. Autom. Control, 59 (2014), pp.~2353--2368.

\bibitem{cover}
{\sc T.M. Cover and J.A. Thomas}, {\em Elements of information theory}, John
  Wiley and Sons, Inc., 1991.

\bibitem{dynkin79}
{\sc E.~B. {Dynkin} and A.~A. {Yushkevich}}, {\em Controlled {M}arkov
  processes}, Springer-Verlag, New York, 1979.

\bibitem{Flynn1976}
{\sc J.~{Flynn}}, {\em Conditions for the equivalence of optimality criteria in
  dynamic programming}, Ann. Statist., 4 (1976), pp.~936--953.

\bibitem{Flynn1980}
\leavevmode\vrule height 2pt depth -1.6pt width 23pt, {\em On optimality
  criteria for dynamic programs with long finite horizons}, J. Math. Anal.
  Appl., 76 (1980), pp.~202--208.

\bibitem{hern1996discrete}
{\sc O.~{Hernandez-Lerma} and J.~B. {Lasserre}}, {\em Discrete-time {M}arkov
  control processes: Basic optimality criteria}, no.~v. 1 in Applications of
  Mathematics Stochastic Modelling and Applied Probability, Springer Verlag,
  1996.

\bibitem{Lerma-Lasserre1997}
\leavevmode\vrule height 2pt depth -1.6pt width 23pt, {\em Policy iteration for
  average cost {M}arkov control processes on {B}orel spaces}, Acta Applicandae
  Mathematica, 47 (1997), pp.~125--154.

\bibitem{varayia86}
{\sc P.~R. {Kumar} and P.~{Varaiya}}, {\em Stochastic systems: Estimation,
  identification, and adaptive control}, Prentice Hall, 1986.

\bibitem{Meyn1997}
{\sc S.~P. {Meyn}}, {\em The policy improvement algorithm for markov decision
  processes with general state space}, IEEE Trans. Autom. Control, 42 (1997),
  pp.~1663--1680.

\bibitem{puterman94}
{\sc M.~L. {Puterman}}, {\em Markov decision Processes}, Wiley, New York, 1994.

\bibitem{Schal92}
{\sc M.~{Schal}}, {\em On the second optimality equation for semi-{M}arkov
  decision models}, Math. Op. Res., 17 (1992), pp.~470--486.

\bibitem{Sennott95}
{\sc L.~I. {Sennott}}, {\em Another set of conditions for average optimality in
  {M}arkov control processes}, Systems and Control Letters, 24 (1995),
  pp.~147--151.

\bibitem{2014arXiv1402.1009T}
{\sc I.~{Tzortzis}, C.~D. {Charalambous}, and T.~{Charalambous}}, {\em {Dynamic
  Programming Subject to Total Variation Distance Ambiguity}}, ArXiv e-prints,
  (2014).

\bibitem{vanSchuppen10}
{\sc J.~H. {Van Schuppen}}, {\em Mathematical control and system theory of
  discrete-time stochastic systems}, Preprint, 2014.

\bibitem{Yushkevich1973}
{\sc A.~A. {Yushkevich}}, {\em On a class of strategies in general {M}arkov
  decision models}, Theory Probab. Appl., 18 (1973), pp.~777--779.

\end{thebibliography}
\end{document}